\documentclass[11pt]{article} 

\usepackage{amsmath}
\usepackage{amsthm}
\usepackage{amsfonts}
\usepackage{mathrsfs}
\usepackage{stmaryrd}
\usepackage{setspace}
\usepackage{fullpage}
\usepackage{amssymb}
\usepackage{breqn}
\usepackage{enumitem}
\usepackage{bbold} 
\usepackage{authblk}
\usepackage{comment}
\usepackage{hyperref}
\usepackage{pgf,tikz}
\usepackage{graphicx}
\usepackage{xcolor}
\usepackage{caption}
\usetikzlibrary{decorations.pathreplacing,arrows}
\usetikzlibrary{positioning,chains,fit,shapes,calc}
\tikzstyle{vertex}=[circle,draw=black,fill=black,inner sep=0,minimum size=3pt,text=white,font=\footnotesize]
\usetikzlibrary{graphs}
\usetikzlibrary{shapes.symbols}

\newtheorem*{rep@theorem}{\rep@title}
\newcommand{\newreptheorem}[2]{%
\newenvironment{rep#1}[1]{%
 \def\rep@title{#2 \ref{##1}}%
 \begin{rep@theorem}}%
 {\end{rep@theorem}}}
\makeatother

\numberwithin{equation}{section}
\theoremstyle{plain}

\newtheorem*{proposition*}{Proposition}

\newtheorem{thm}{Theorem}[section]
\newtheorem{theorem}{Theorem}[section]

\newtheorem{lemma}[thm]{Lemma}
\newtheorem{conj}[thm]{Conjecture}
\newtheorem{proposition}[thm]{Proposition}
\newtheorem{prop}[thm]{Proposition}

\newtheorem{defn}[thm]{Definition}

\newtheorem{claim}[thm]{Claim}
\newtheorem{obs}{Observation}

\newtheorem*{theorem*}{Theorem}
\newtheorem*{prop*}{Proposition}

\newcommand\ex{\ensuremath{\mathrm{ex}}}

\newcommand\ind{\ensuremath{\mathrm{ind}}}

\newcommand\cB{{\mathcal B}}

\newcommand\cP{{\mathcal P}}

\newcommand{\ignore}[1]{}

\title{Induced planar Tur\'an numbers}

\author{Ervin Gy\H{o}ri\footnote{HUN-REN, Alfr\'ed R\'enyi Institute of Mathematics, Budapest, Hungary. Email:\texttt{gyori@renyi.hu}} \quad and \quad Hilal Hama Karim$^{*}$\footnote{Department of Computer Science and Information Theory, Faculty of Electrical Engineering and Informatics, Budapest University of Technology and Economics, Műegyetem rkp. 3., H-1111 Budapest, Hungary. E-mail: \texttt{hilal.hamakarim@edu.bme.hu}}}

\date{}

\begin{document}

\maketitle

\begin{abstract}
    The planar Tur\'an number of a graph $F$ is the maximum number of edges an $n$-vertex $F$-free planar graph can have. We study the case where $F$ is forbidden as an induced subgraph, thereby introducing the induced planar Tur\'an numbers. We will determine a sharp upper bound when $F$ is $\Theta_4$, a $4$-cycle with a diagonal edge, and obtain exact extremal values in case $F$ is a path $P_k$ on $k$ vertices, for $k=3,4$ and $5$.
\end{abstract}

\section{Introduction}
   The Tur\'an number of a graph $F$, $\ex(n,F)$, is the maximum possible number of edges in a graph on $n$ vertices that does not contain $F$ as a (not necessarily induced) subgraph. Ever since the seminal work of Tur\'an \cite{turan41} in 1941, in which he determined the Tur\'an number of $K_k$, for every $k\geq 3$, this has been a major area of research in extremal combinatorics. In fact, if $F$ is not a clique and one desires to forbid $F$ as an induced subgraph, then the answer is trivially $\binom{n}{2}$, since complete graphs always avoid induced copies of $F$. Nevertheless, induced Tur\'an numbers have been introduced and studied through adding various conditions and restrictions rather than the plausible way. We refer the interested reader to  
   \cite{Axenovichzim2025, PromelSteger1991,PromelSteger1992,PromelSteger1993,Lohtaitimzh2018}. On the other hand, in the planar variant of the Tur\'an problem, it turned out that requiring the forbidden subgraph to be induced makes a non-trivial and interesting problem.
   
   The planar Tur\'an number of $F$, denoted by $\ex_\cP(n,F)$, is the maximum possible number of edges in an $n$-vertex planar graph that does not contain $F$ as a subgraph. This was initiated by Dowden \cite{Dowden2016} in 2016, who determined sharp upper bounds for $\ex_\cP(n, C_4)$ and $\ex_\cP(n, C_5)$. This variant of the classical Tur\'an problem already offers many interesting and challenging problems, and has been studied for many classes of graphs, for example, cycles \cite{CLidiLiShan2022, GhGyMPX, GyVarZhu2024, ShiWalshYu2025}, paths \cite{LanShi2019}, theta graphs \cite{GuGyLuWaYa2026, LanShiSongTheta2019, XiaoGoGyPauZmo2024}.
   In this paper we study the planar Tur\'an problem where the forbidden subgraph $F$ is required to be induced, we will call it the \textit{induced planar Tur\'an number} of $F$ and denote it by $\ex_\cP(n,F^\ind)$. For many graphs $F$, we can have a maximal planar graph, which we call a triangulation, that does not contain $F$ as an induced subgraph. However, this is still not as trivial as the non-planar case, since in general there are many different triangulations on $n$ vertices, many of which may not avoid an induced $F$. 
   
   Let us define $\ex_\cP(n, F^\ind)$ precisely. Let $F^\ind \subseteq G$ denote that $G$ contains an induced subgraph isomorphic to $F$. Similarly,  and use $F^\ind \nsubseteq G$ to denote that $G$ does not contain any induced subgraph isomorphic to $F$. In the latter case,  we say $G$ is $F^\ind$-free, or $G$ does not contain an induced copy of $F$. Then,
   $$\ex_\cP(n,F^\ind):= \max \{|E(G)| : \ G \text{ is an} \ n\text{-vertex  planar graph and } F^\ind\text{-free}\}$$
   

Throughout the paper, we denote by $K_n$ and $\overline{K}_n$ the complete and empty graphs on $n$ vertices, respectively. For two graphs $H$ and $G$,  $H \cup G$ denotes their disjoint union and $H+G$ denotes the graph obtained from $H \cup G$ by adding all the edges between $H$ and $G$. Let $tH$ denote the disjoint union of $t$ copies of $H$. Let $P_k$ and $C_k$ denote a path and a cycle on $k$ vertices, respectively.


Obviously, if $G$ is $F$-free, then it is $F^\ind$-free, and hence $\ex_\cP(n,F^\ind) \geq \ex_\cP(n,F)$, for any graph $F$. Observe that characterizing those graphs for which the planar Tur\'an number is $3n-6$ is an interesting and challenging problem, see \cite{LanShiSong2019} for some sufficient conditions. Here, in the induced version, we provide some sufficient conditions, which we believe to be an exhaustive list.

For complete graphs, any copy is an induced copy. Then $\ex_\cP(n, K_3^\ind)=\ex_\cP(n,K_3)=2n-4$ and $\ex_\cP(n, K_4^\ind)=\ex_\cP(n,K_4)=3n-6$, for all $n\geq 6$. Also, it is easy to see that the triangulation $K_2+P_{n-2}$ does not contain any induced cycle of length $k \geq 4$, which implies that $\ex_\cP(n,C_k^\ind)=3n-6$, for every $k \geq 4$.
Note that the same triangulation does not contain an induced copy of two vertex-disjoint triangles, implying that $\ex_\cP(n,(2K_3)^\ind)=3n-6$. 

Let $S_k$ be a star with $k$ leaves (i.e. it has $k+1$ vertices). Plummer \cite{plummer1990} characterized maximal planar graphs with no induced $S_3$ (which are called claw-free). In particular, for every $n\geq 3$, there are $S_k^\ind$-free triangulations on $n$ vertices, and hence, $\ex_\cP(n,S_3^\ind)=3n-6$. 

Thus, the following proposition follows.

\begin{proposition}\label{3n-6}
    For every $n\geq 6$, $\ex_\cP(n,F^\ind)=3n-6$, if $F$ is one of the following graphs.
    
    \textbf{(i)} $C_k$, for any $k\geq 4$, \quad \textbf{(ii)} $S_3$, \quad \textbf{(iii)} $K_4$, or \quad \textbf{(iv)} $2K_3$.
\end{proposition}

Clearly, if $F^\ind \subseteq H$, then $\ex_\cP(n, H^\ind) \geq \ex_\cP(n, F^\ind)$. Therefore, the following stronger version of Proposition \ref{3n-6} follows immediately.

\begin{proposition}\label{suffconditions}
    For any graph $F$, and every $n\geq 3$, $\ex_\cP(n,F^\ind)=3n-6$ whenever one of the following holds:
    \begin{itemize}
        \item[(i)] $F$ contains a $K_4$.
        \item[(ii)] $F$ contains an induced $C_k$, for $k\geq 4$ (i.e. $F$ is not chordal).
        \item[(iii)] $F$ contains an induced $S_k$, for $k\geq 3$ (i.e. $F$ is not claw-free).
        \item[(iv)] $F$ contains an induced copy of $2K_3$.
    \end{itemize}
\end{proposition}

The full characterization of graphs $F$ for which $\ex_\cP(n, F^\ind) = 3n-6$ remains an interesting open question. We conjecture that the converse of Proposition \ref{suffconditions} is also true, that is, $\ex_\cP(n,F^\ind)=3n-6$ if and only if one of the the above four conditions holds for $F$.

Note that any complete bipartite graph on $n\geq 4$ vertices contains either an induced $S_3$ or an induced $C_4$, and hence, by proposition \ref{suffconditions}, the induced planar Tur\'an number for them is also $3n-6$.

Theta graphs $\Theta_k$ are those consist of a cycle $C_k$ with one extra edge joining two non-consecutive vertices on the cycle. 
Using the concept of triangular blocks, a contribution method introduced by Ghosh, Gy\H{o}ri, Martin, Paulos and Xiao \cite{GhGyMPX}, we determine an upper bound for $\ex_\cP(n, \Theta_4^\ind)$ that is sharp for infinitely many values of $n$.

\begin{theorem}\label{theta4}
    For every $n\geq 5$, $\ex_\cP(n,\Theta_4^\ind)\leq \frac{8}{3}(n-2)$. Moreover, this bound is sharp for every $n\equiv 20 \ (mode \ 36)$.
\end{theorem}

Note that for every $k\geq 5$, any $\Theta_k$ contains an induced cycle of length at least $4$. Then, by Proposition \ref{suffconditions}, $\ex_\cP(n, \Theta_k^\ind)=3n-6$, for every $k\geq 5$. Notice that $\Theta_k$ is a collections of graphs that contains more than one graph if $k\geq 6$, but still any one of them contains an induced cycle of length at least $4$. Hence, whether we forbid the whole collection or any single member of $\Theta_k$, the extremal value is still $3n-6$.




Next we will consider paths, which are probably the most interesting class of graphs to deal with. Even in the non-induced version, $\ex_\cP(n,P_k)$ is known only for $k \leq 11$ \cite{LanShi2019}. Observe also that if a tree contains a vertex of degree three, then it contains an induced $S_3$, which means that paths are the only non-trivial trees to consider in the induced setting. We determine the exact value for induced paths  on $3, 4$  and $5$ vertices. 

\begin{theorem}\label{P3} For every $n \geq 1$,
$\ex_\cP(n,P_3^\ind)=\left\{
\begin{array}{lll}
      \frac{6}{4}n & , & n \equiv 0 \ (\text{mode }4) \\
      \frac{6}{4}(n-1) & , & n \equiv 1 \ (\text{mode }4) \\
      \frac{6}{4}(n-2)+1 & , & n\equiv 2 \ (\text{mode }4)\\
      \frac{6}{4}(n-3)+3 & , & n \equiv 3 \ (\text{mode }4)\\
      \end{array} \right.$

      Moreover, the unique extremal graph is the disjoint union of  $\lfloor \frac{n}{4}\rfloor$ copies of $K_4$ and maybe a smaller clique on the remaining vertices.
\end{theorem}



\begin{theorem}\label{P_4} For every $n \geq 3$,

    $\ex_\cP(n,P_4^\ind)=\left\{
\begin{array}{lll}
      3n-6 & , & n\leq 6 \\
      \lfloor\frac{8}{3}(n-2)\rfloor+1 & , & n\geq 7 \\
      \end{array} \right.$
\end{theorem}

Finally, we have what we consider as the major result of the paper. 
\begin{theorem}\label{P_5} For every $n \geq 3$,

    $\ex_\cP(n,P_5^\ind)=\left\{
\begin{array}{lll}
      3n-6 & , & n\leq 8 \\
      \lfloor\frac{11n}{4}\rfloor-4 & , & n\geq 9 \\
      \end{array} \right.$
\end{theorem}

Throughout, we use the following notation.  Given a vertex $v \in G$, we denote by $N_G(v)$ the set of the neighbors of $v$ in $G$, and $d_G(v):=|N_G(v)|$. We drop the indices whenever they are clear from the context. Also, for a subset $X \subseteq V(G)$, $N_X(v):= N_G(v) \cap X$ and $d_X(v):=|N_X(v)|$. For a graph $G$, $\delta(G)$ denotes the minimum degree in $G$ and $\Delta(G)$ denotes the maximum degree in $G$. Also, $G[X]$ denotes the subgraph of $G$ induced by $X$, and $G\setminus X:=G[V(G) \setminus X]$. When $X=\{x\}$, we simply write $G \setminus x$ for $G \setminus X$.

\section{Graphs with no induced $\Theta_4$}

 In this section, we prove Theorem \ref{theta4}. By a common neighbor of an edge, we mean a common neighbor of its endpoints. If an edge has two non-adjacent common neighbors, then they form an induced $\Theta_4$. Thus, in a graph $G$ that does not contain an induced $\Theta_4$, the common neighbors of each edge must form a clique. Therefore, if an edge has three common neighbors, they form a $K_5$, and $G$ is not planar. Hence, the following Lemma follows.  

 \begin{lemma}\label{no_edge_has_3_com_nbhd}
     Let $G$ be a planar graph that does not contain an induced $\Theta_4$. Then 
     \begin{enumerate}
         \item Every edge of $G$ has at most two common neighbors.

         \item Any edge with two common neighbors belongs to one and only one $K_4$ in $G$.
     \end{enumerate}
     
 \end{lemma}

For the proof of Theorem \ref{theta4} 
we use the concept of \textit{triangular blocks} introduced in \cite{GhGyMPX}. Let us first recall the definition and some relevant properties.

\begin{defn} \cite{GhGyMPX} Let $G$ be a plane graph. An edge $e \in E(G)$ is a triangular-block if it is not in any face of length 3 (this is called a \textbf{trivial triangular block}), otherwise we inductively build up the block; start with the subgraph $H:=e$, keep adding to $H$ all faces of length 3 (and their edges) that contain an edge of $H$ until no such is left.
\end{defn}

Let $\cB$ be the set of all triangular blocks of $G$.
It is clear that every edge of $G$ is in exactly one triangular block. So,
\[e(G)=\sum_{B \in \cB} e(B). \]
 
Observe that for a block $B$, each face, except possibly the outer face, is a triangle. The edges on the boundary face of $B$ are called \textit{external} edges and others are called \textit{internal} edges. The faces of $B$ that are bounded only by edges of $G$ are called \textit{internal} faces of $B$. Each face of $G$ that has an external edge of $B$ on its boundary is called an \textit{external} face of $B$. Note that each external edge of $B$ belongs to an internal face of $B$ and an external face of $B$. Any external face of $B$ must have length at least four, otherwise it should have been added to $B$, by definition.

Let $f(G)$ denote the number of faces of $G$. For each triangular block $B$ of $G$, let $f(B)$ denote the contribution of $B$ to $f(G)$, defined as follows. If an external edge $e$ of $B$ is on the boundary of an external face of length $\ell$, then the contribution of $e$ is $c(e)=1/\ell$. Then, $f(B)=\sum_{e \text{ is external edge}} c(e)+ \#$ of internal faces of $B$. Then, $$f(G)=\sum_{B\in \cB}f(B).$$


\begin{prop}\label{tr-blocksinTheta4-free} Let $G$ be a plane graph without containing an induced $\Theta_4$. Then, each triangular block $B$ of $G$ has at most $4$ vertices, and it is one of the following graphs:
    
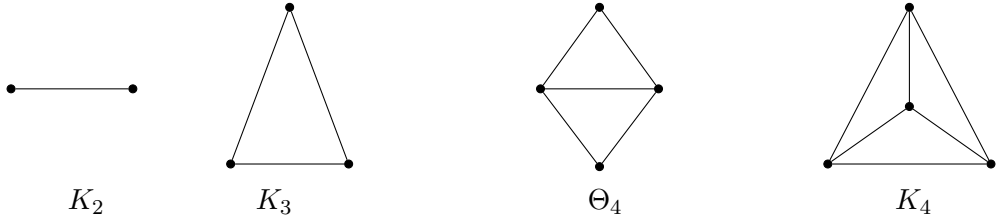
\begin{center}
    \begin{tikzpicture}
    \node[vertex] (a) {};
\node [vertex] (b) [right=1.5cm of a] {};
\node [vertex] (c) [above right=1cm and 2cm of b] {};
\node [vertex] (d) [below left=2cm and 0.7cm of c] {};
\node [vertex] (e) [below right=2cm and 0.7cm of c] {};
\node [vertex] (f) [right=4cm of c] {};
\node[vertex] (g) [below left=1cm and 0.7cm of f] {};
\node[vertex] (h) [below right=1cm and 0.7cm of f] {};
\node[vertex] (i) [below=2cm of f] {};
\node [vertex] (j) [right=4cm of f] {};
\node [vertex] (k) [below left=2cm and 1cm of j] {};
\node [vertex] (l) [below right=2cm and 1cm of j] {};
\node [vertex] (m) [below=1.2cm of j] {};

\node[] at (1,-1.5) {$K_2$};
\node[] at (3.5,-1.5) {$K_3$};
\node[] at (7.9,-1.5) {$\Theta_4$};
\node[] at (12,-1.5) {$K_4$};

\path (a) edge (b);
\path (c) edge (d);
\path (c) edge (e);
\path (d) edge (e);
\path (f) edge (g);
\path (f) edge (h);
\path (g) edge (h);
\path (g) edge (i);
\path (i) edge (h);
\path (j) edge (k);
\path (j) edge (l);
\path (j) edge (m);
\path (k) edge (l);
\path (k) edge (m);
\path (l) edge (m);
\end{tikzpicture}
 \end{center}   
\end{prop}

\begin{proof}
    From the definition it is easy to see that each $k$-vertex triangular block is built from a $(k-1)$-block by adding a vertex which forms a $3$-face with an edge of the $(k-1)$-block. Then, clearly on two and three vertices we have an edge and a triangle, respectively. A $4$-vertex block is formed when an edge of a triangle is in a $3$-face, and this immediately gives a $\Theta_4$. Let $a,b,c,d$ be the vertices of the $\Theta_4$ and $ab,bc,cd,da, ac$ be its edges. If $bd$ is not an edge, this will be an induced $\Theta_4$, a contradiction. Thus, $bd$ is also an edge. If none of $abd$ and $cbd$ is a face in $G$, then $bd$ is not added to the block, and the $4$-vertex block remains to be the $\Theta_4$. If any of $abd$ or $cbd$ is a face of $G$, then we add $bd$ and the block will be a $K_4$. This shows that any block on at most $4$-vertices is one of $K_2, K_3, \Theta_4$ and $K_4$.
    
    Suppose that $B$ is a triangular block with $5$ vertices. Then, $B$ is built from either a $\Theta_4$ or a $K_4$. It means that an edge of the $\Theta_4$ or the $K_4$ is in a $3$-face with the fifth vertex of the block. Note that as explained in the previous paragraph, the $\Theta_4$ block also sits in a $K_4$. However, if an edge of a $K_4$ lies in a triangle (face or not) with a vertex not in the $K_4$, then that edge will have three common neighbors, contradicting Lemma \ref{no_edge_has_3_com_nbhd}. Therefore, there is no triangular block on $5$ vertices.
\end{proof}


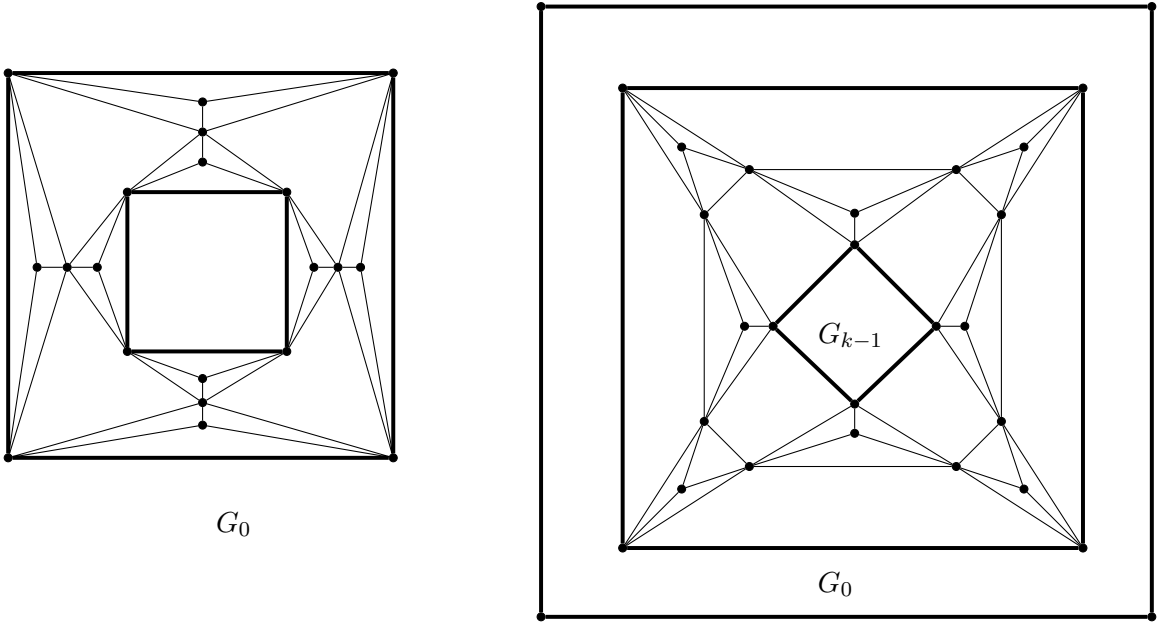
\begin{figure}[h]
\begin{center}
    \begin{tikzpicture}
    \node[vertex] (a) {};
\node[vertex] (b) [right=5cm of a] {};
\node[vertex] (c) [below=5cm of b] {}; 
\node[vertex] (d) [below=5cm of a] {}; 

\node[vertex] (x) [below right=1.5cm and 1.5cm of a] {};
\node[vertex] (y) [right=2cm of x] {}; 
\node[vertex] (z) [below=2cm of y] {}; 
\node[vertex] (w) [below=2cm of x] {};

\node[vertex] (a1) [below right=0.3cm and 2.5cm of a] {};
\node[vertex] (a2) [below right=0.7cm and 2.5cm of a] {};
\node[vertex] (a3) [below right=1.1cm and 2.5cm of a] {};

\node[vertex] (b1) [below right=2.5cm and 3.98cm of a] {};
\node[vertex] (b2) [below right=2.5cm and 4.3cm of a] {};
\node[vertex] (b3) [below right=2.5cm and 4.6cm of a] {};

\node[vertex] (c1) [below right=3.98cm and 2.5cm of a] {};
\node[vertex] (c2) [below right=4.3cm and 2.5cm of a] {};
\node[vertex] (c3) [below right=4.6cm and 2.5cm of a] {};

\node[vertex] (d1) [below right=2.5cm and 0.3cm of a] {};
\node[vertex] (d2) [below right=2.5cm and 0.7cm of a] {};
\node[vertex] (d3) [below right=2.5cm and 1.1cm of a] {};

\node[] at (3,-6) {$G_0$};

\draw[line width=1.5pt] (a) -- (b) -- (c) -- (d) -- (a);
\draw[line width=1.5pt] (x) -- (y) -- (z) -- (w) -- (x);

\draw (a) -- (a1) -- (b) -- (a2) -- (a);
\draw (x) -- (a3) -- (y) -- (a2) -- (x);
\draw (a1) -- (a2) -- (a3);

\draw (y) -- (b1) -- (z) -- (b2) -- (y);
\draw (b) -- (b3) -- (c) -- (b2) -- (b);
\draw (b1) -- (b2) -- (b3);

\draw (z) -- (c1) -- (w) -- (c2) -- (z);
\draw (c) -- (c3) -- (d) -- (c2) -- (c);
\draw (c1) -- (c2) -- (c3);

\draw (a) -- (d1) -- (d) -- (d2) -- (a);
\draw (x) -- (d3) -- (w) -- (d2) -- (x);
\draw (d1) -- (d2) -- (d3);

\node[vertex] (A) [above right=0.8cm and 7cm of a] {};
\node[vertex] (B) [right=8cm of A] {};
\node[vertex] (C) [below=8cm of B] {}; 
\node[vertex] (D) [below=8cm of A] {}; 

\node[vertex] (X) [below right=1cm and 1cm of A] {};
\node[vertex] (Y) [right=6cm of X] {}; 
\node[vertex] (Z) [below=6cm of Y] {}; 
\node[vertex] (W) [below=6cm of X] {};

\draw[line width=1.5pt] (A) -- (B) -- (C) -- (D) -- (A);
\draw[line width=1.5pt] (X) -- (Y) -- (Z) -- (W) -- (X);
\node[] at (11,-6.8) {$G_0$};

\node[vertex] (xx) [below right=2cm and 3cm of X] {};
\node[vertex] (yy) [below right=1cm and 1cm of xx] {}; 
\node[vertex] (zz) [below=2cm of xx] {}; 
\node[vertex] (ww) [below left =1cm  and 1cm of xx] {};

\draw[line width=1.5pt] (xx) -- (yy) -- (zz) -- (ww) -- (xx);
\node[] at (11.2,-3.5) {$G_{k-1}$};

\node[vertex] (aa1) [below right=1cm and 1.6cm of X] {};
\node[vertex] (aa2) [below right=1.6cm and 1cm of X] {};

\node[vertex] (aa3) [below right=0.7cm and 0.7cm of X] {};

\node[vertex] (bb1) [below left=1cm and 1.6cm of Y] {};
\node[vertex] (bb2) [below left=1.6cm and 1cm of Y] {};

\node[vertex] (bb3) [below left=0.7cm and 0.7cm of Y] {};

\node[vertex] (cc1) [above left=1cm and 1.6cm of Z] {};
\node[vertex] (cc2) [above left=1.6cm and 1cm of Z] {};

\node[vertex] (cc3) [above left=0.7cm and 0.7cm of Z] {};

\node[vertex] (dd1) [above right=1cm and 1.6cm of W] {};
\node[vertex] (dd2) [above right=1.6cm and 1cm of W] {};

\node[vertex] (dd3) [above right=0.7cm and 0.7cm of W] {};

\node[vertex] (xx1) [above=0.3cm of xx] {};
\node[vertex] (yy1) [right=0.26cm of yy] {};
\node[vertex] (zz1) [below=0.27cm of zz] {};
\node[vertex] (ww1) [left=0.26cm of ww] {};

\draw (X) -- (aa3);
\draw (X) -- (aa1) -- (aa3) -- (aa2) -- (X);
\draw (Y) -- (bb3);
\draw (Y) -- (bb1) -- (bb3) -- (bb2) -- (Y);
\draw (Z) -- (cc3);
\draw (Z) -- (cc1) -- (cc3) -- (cc2) -- (Z);
\draw (W) -- (dd3);
\draw (W) -- (dd1) -- (dd3) -- (dd2) -- (W);

\draw (aa1) -- (bb1) -- (bb2) -- (cc2) -- (cc1) -- (dd1) -- (dd2) -- (aa2) -- (aa1);

\draw (xx) -- (xx1);
\draw (xx) -- (aa1) -- (xx1) -- (bb1) -- (xx);
\draw (yy) -- (yy1);
\draw (yy) -- (bb2) -- (yy1) -- (cc2) -- (yy);
\draw (zz) -- (zz1);
\draw (zz) -- (cc1) -- (zz1) -- (dd1) -- (zz);
\draw (ww) -- (ww1);
\draw (ww) -- (aa2) -- (ww1) -- (dd2) -- (ww);

\end{tikzpicture}
 \end{center}   
    \caption{Illustration of constructing $G_k$}
    \label{figrTheta4}
\end{figure}

\begin{proof}[\textbf{Proof of Theorem \ref{theta4}}]
Let $G$ be a planar graph on $n$ vertices that does not contain an induced $\Theta_4$. To prove the upper bound, it is enough to show that for each triangular block $B$, we have
\begin{equation}\label{8f5e}
    8f(B)-5e(B)\leq 0
\end{equation}
Since it follows that $$8f(G)-5e(G)=8\sum_{B \in \cB}f(B)-5\sum_{B\in \cB}e(B)=\sum_{B \in \cB}(8f(B)-5e(B))\leq 0$$
Then, together with Euler's formula, $n+f=e+2$, we obtain that $e \leq \frac{8}{3}(n-2)$.

Let $B$ be a triangular block of $G$. By the Proposition \ref{tr-blocksinTheta4-free}, $B$ is one of the graphs  $K_2, K_3, \Theta_4$ and $K_4$.

\smallskip
If $B$ is the trivial block $K_2$, then $e(B)=1$, and the faces containing the edge is at least a $4$-face, otherwise it would not be a trivial block. Thus, $f(B)\leq 2 \cdot \frac{1}{4}=1/2$. Hence, \ref{8f5e} clearly holds.

\smallskip
If $B=K_3$, then $e(B)=3$, and each of its edges is an external edge. If each edge belongs to a different external face, then $f(B)\leq 3 \cdot \frac{1}{4}+1\leq \frac{5}{4}$. Thus, $8f(B)-5e(B)=-5<0$. If two edges of $B$ are in one external face, then it must have length at least $5$, otherwise we would have an induced $\Theta_4$. Then $f(B)\leq \frac{1}{5}+\frac{1}{5}+\frac{1}{4}+1$. Again, $8f(B)-5e(B)<0$.

\smallskip
If $B=\Theta_4$, then $e(B)=5$. Let $V(B)=\{a,b,c,d\}$ with edges $ab,bc,cd,da$ and $ac$. Recall that $bd$ is an edge of $G$. (Note that in this case the edge $bd$ must be a trivial block).

Now, $B$ has four exterior edges, if each of them is in a different face then $f(B)\leq 4\cdot\frac{1}{4}+2=3$. Then, $8f(B)-5e(B)<0$.

The two edges $ab$ and $bc$ together cannot be on the boundary of an external face, because of the edge $bd$. For the same reason, $cd$ and $da$ together cannot be on the boundary of an external face. Thus, if the two edges $ab$ and $da$ are on the boundary of one external face, then the length of such a face must be at least $5$. Since a face of length four means a common neighbor $x$ for $b$ and $d$. Note that $x \neq c$, since the face cannot be $abcd$, because of the edge $bd$. But then the edge $bd$ will have three common neighbors $a, c$ and $x$, which contradicts Lemma \ref{no_edge_has_3_com_nbhd}. Then, $f(B)\leq 2\cdot \frac{1}{5}+\frac{1}{4}+\frac{1}{4}+2$, and then $8f(B)-5e(B)<0$.

\smallskip
Finally, if $B=K_4$, then $e(B)=6$. If each boundary edge is in a different external face, then $f(B)=3 \cdot \frac{1}{4}+4+3$, and $8f(B)-5e(B)=0$. assume that two of the boundary faces of $B$ are in one external face. Let the vertices on the boundary face of $B$ be $a, b$ and $c$, and $ab$ and $bc$ are in one external face. If the external face face has length $4$, then there is a common neighbor $x$ for the edge $ac$ different from two other vertices of $B$, that is, $ac$ has three common neighbors, a contradiction. Thus the external face has length at least $5$. Thus, $f(B)=2 \cdot \frac{1}{5}+\frac{1}{4}+3$, and $8f(B)-5e(B)<0$. This completes the proof of the upper bound. 

\smallskip
Now, for each $n=36k+20$, for non-negative integers $k$, we recursively construct an $n$-vertex planar graph $G$ 
such that every block is a $K_4$ and each external edge of the block is contained in a distinct external face of length four. Then, $G$ does not contain an induced $\Theta_4$ and along the above proof we have $8f(G)-5e(G)=0$, which implies that $e(G)=\frac{8}{3}(n-2)$. Let $G_0$ be the given graph on the left side of Figure \ref{figrTheta4}. For $k\geq 1$, assume that we have constructed $G_{k-1}$, then we construct $G_k$ as illustrated on the right side of Figure \ref{figrTheta4}: Put $G_{k-1}$ inside the inner 4-face of $G_0$, and between the inner face of $G_0$ and the outer face of $G_{k-1}$ add $16$ vertices and add edges as drawn in the figure. 
\end{proof}

\section{Graphs with no induced short paths}

In this section we prove Theorem \ref{P3} and Theorem \ref{P_4}.
\begin{proof}[\textbf{Proof of Theorem \ref{P3}}]
    Let $G$ be an $n$-vertex planar graph with no induced $P_3$. Every vertex of degree at least two can be the middle (center) vertex of an induced $P_3$ unless every pair of its neighbors are adjacent. Thus, every vertex $v$ lies in a clique of order $d(v)+1$. Since there is no $K_5$ in a planar graph, then $d(v) \leq 3$, for every vertex of $G$. The above argument further implies that the vertices of degree three in $G$ form disjoint $K_4$'s. Clearly, if $n$ is not divisible by $4$, the fewer the number of vertices of degree less than three the more edges we would have in the graph. It is easy to see that the maximum number of edges is obtained by $\lfloor\frac{n}{4}\rfloor$ disjoint copies of $K_4$ and a smaller clique on the remaining vertices. This completes the proof.
\end{proof}

\begin{proof}[\textbf{Proof of Theorem \ref{P_4}}]
    For each $n\leq 5$, there is a unique triangulation, and it does not contain an induced $P_4$. For $n=6$, the triangulation $\overline{K}_2+C_4$ does not contain an induced $P_4$. This proves the first part. 
    Now, we assume that $n\geq 7$. For the lower bound, consider the following construction. Write $n-2=3a+b$, for some integers $a\geq 1$ and $0\leq b<3$. 
    Let $H_{n,4}:=(aP_3 \cup P_b)+K_2$. It is easy to see that $H_{n,4}$ is a $P_4^\ind$-free planar graph with $|E(H_{n,4})|= \lfloor \frac{8}{3}(n-2)\rfloor +1$.

    
    To prove the upper bound, let $G$ be a planar graph of $n \geq 7$ vertices that does not contain an induced $P_4$. 

    If $n=7$, there are five (up to isomorphism) triangulations, and one can check that each of them contains an induced $P_4$, see Figure \ref{fig:7trhasP4}. Thus, for $n=7$, there are at most $3n-7=14=\lfloor \frac{8}{3}(n-2)\rfloor+1$ edges, and we are done.

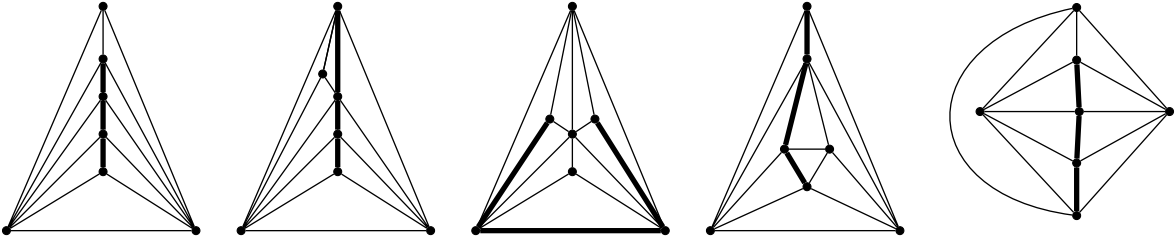
\begin{figure}[h]
        \centering
\begin{center}
    \begin{tikzpicture}
    \node[vertex] (a1) {};
    \node[vertex] (b1) [right=2.4cm of a1] {};
\node [vertex] (c1) [above right=0.7cm and 1.2cm of a1] {};
\node [vertex] (d1) [above right=1.2cm and 1.2cm of a1] {};
\node [vertex] (e1) [above right=1.7cm and 1.2cm of a1] {};
\node [vertex] (f1) [above right=2.2cm and 1.2cm of a1] {};

\node [vertex] (g1) [above right=2.9cm and 1.2cm of a1] {};

\draw[line width=0.5pt] (a1) -- (b1) -- (c1)-- (a1)-- (d1) -- (b1) -- (e1) -- (a1) -- (f1) -- (b1) -- (g1) -- (a1);
\draw[line width=0.5pt] (c1) -- (d1)-- (e1) -- (f1) -- (g1);
\draw[line width=2pt] (c1) -- (d1)-- (e1) -- (f1);



\node[vertex] (a2) [right=3cm of a1] {};
    \node[vertex] (b2) [right=2.4cm of a2] {};
\node [vertex] (c2) [above right=0.7cm and 1.2cm of a2] {};
\node [vertex] (d2) [above right=1.2cm and 1.2cm of a2] {};
\node [vertex] (e2) [above right=1.7cm and 1.2cm of a2] {};
\node [vertex] (f2) [above right=2cm and 1cm of a2] {};

\node [vertex] (g2) [above right=2.9cm and 1.2cm of a2] {};

\draw[line width=0.5pt] (a2) -- (b2) -- (c2)-- (a2)-- (d2) -- (b2) -- (e2) -- (a2) -- (f2) -- (g2) -- (b2);
\draw[line width=0.5pt] (a2) -- (g2) -- (e2);
\draw[line width=0.5pt] (c2) -- (d2)-- (e2) -- (f2) -- (g2);
\draw[line width=2pt] (c2) -- (d2)-- (e2) -- (g2);


\node[vertex] (a3) [right=3cm of a2] {};
    \node[vertex] (b3) [right=2.4cm of a3] {};
\node [vertex] (c3) [above right=0.7cm and 1.2cm of a3] {};
\node [vertex] (d3) [above right=1.2cm and 1.2cm of a3] {};
\node [vertex] (e3) [above right=1.4cm and 1.5cm of a3] {};
\node [vertex] (f3) [above right=1.4cm and 0.9cm of a3] {};

\node [vertex] (g3) [above right=2.9cm and 1.2cm of a3] {};

\draw[line width=0.5pt] (a3) -- (b3) -- (c3)-- (a3)-- (d3) -- (b3) -- (e3) -- (d3) -- (f3) -- (a3) -- (g3) -- (b3);
\draw[line width=0.5pt] (f3) -- (g3) -- (e3);
\draw[line width=0.5pt] (c3) -- (d3)-- (g3);
\draw[line width=2pt] (f3) -- (a3)-- (b3) -- (e3);


\node[vertex, right=3cm of a3] (a4) {};
    \node[vertex] (b4) [right=2.4cm of a4] {};
\node [vertex] (c4) [above right=0.5cm and 1.2cm of a4] {};
\node [vertex] (d4) [above right=1cm and 0.9cm of a4] {};
\node [vertex] (e4) [above right=1cm and 1.5cm of a4] {};
\node [vertex] (f4) [above right=2.2cm and 1.2cm of a4] {};

\node [vertex] (g4) [above right=2.9cm and 1.2cm of a4] {};

\draw[line width=0.5pt] (a4) -- (b4) -- (f4)-- (a4)-- (g4) -- (b4);
\draw[line width=0.5pt] (g4) -- (f4)-- (d4) -- (c4) -- (e4) -- (d4) -- (a4) -- (c4) -- (b4) -- (e4) -- (f4);
\draw[line width=2pt] (c4) -- (d4)-- (f4) -- (g4);


\node[vertex, above right=1.5 and 3.5cm  of a4] (a5) {};
    \node[vertex] (b5) [right=2.4cm of a5] {};
    \node [vertex] (c5) [below right=1.3cm and 1.2cm of a5] {};
\node [vertex] (d5) [below right=0.6cm and 1.2cm of a5] {};
\node [vertex] (e5) [right=1.2cm of a5] {};
\node [vertex] (f5) [above right=0.6cm and 1.2cm of a5] {};

\node [vertex] (g5) [above right=1.3cm and 1.2cm of a5] {};

\draw[line width=0.5pt] (a5) -- (c5) -- (b5)-- (d5)-- (a5) -- (e5) -- (b5) -- (f5) -- (a5) -- (g5) -- (b5);
\draw[line width=0.5pt] (g5) -- (f5)-- (e5) -- (d5) -- (c5);
\draw[line width=0.5pt] (g5) .. controls (12,2.5) and (12,0.5) .. (c5);
\draw[line width=2pt] (c5) -- (d5)-- (e5) -- (f5);

\end{tikzpicture}
 \end{center}   
        \caption{All $7$-vertex triangulations contain induced $P_4$.}
        \label{fig:7trhasP4}
    \end{figure}

    For the rest of the proof, assume that $n\geq 8$.
    We may assume that $\delta(G) \geq 3$, otherwise we are done by induction. In particular, it has no component of order at most $3$. If $G$ has two components of order at most $6$, say $4\leq k_1,k_2 \leq 6$, then they give at most $(3k_1-6)+(3k_2-6)=3(k_1+k_2)-12$ edges to $G$. Replace them by a copy of $H_{k_1+k_2,4}$ which gives $\lfloor\frac{8}{3}(k_1+k_2-2)\rfloor+1$. Since $k_1+k_2\leq 12$, the resulted graph is still $P_4^\ind$-free and has strictly more edges than $G$, and hence, we may assume that $G$ has at most one component of order at most $6$. Let $C$ be such a component, that is, $4 \leq |V(C)|=k\leq 6$, and $G\setminus C$ has no component of order at most $6$ (in particular, $n-k\geq 7$). Then, by induction hypothesis, we have $e(G)\leq 3k-6+ \lfloor\frac{8}{3}(n-k-2)\rfloor+1 < \lfloor\frac{8}{3}(n-2)\rfloor+1$, since $k\leq 6$, and we are done. Thus, we may assume that each component of $G$ has order at least $7$. Then, if $G$ is not connected, we may apply induction on its components (note that the extremal value has a negative additive constant). Hence, we may assume that $G$ is connected.

    Let $u$ be a vertex of maximum degree in $G$. Note that $\Delta(G)\geq 5$, since otherwise $e(G) \leq 2n < \lfloor\frac{8}{3}(n-2)\rfloor+1$, for every $n\geq 8$, a contradiction.
    
    Let $W:= V(G) \setminus (\{u\} \cup N(u))$. We distinguish two cases.

    \vskip3mm
    
    \textbf{Case 1.} $W \neq \emptyset$.
    
    If there are no edges between $W$ and $N(u)$, then $G$ is not connected, a contradiction. Let $v \in W$ be adjacent to some vertex in $N(u)$. Then, $v$ is not adjacent to $u$ but has common neighbors with $u$. Let $w \in N(v) \cap N(u)=:L$. If there is a vertex $u' \in N(u) \setminus N(v)$ such that $u'w \notin E(G)$, then $u'uwv$ is an induced $P_4$, a contradiction. Thus, every vertex in $N(u) \setminus N(v)$ must be adjacent to $w$. Similarly, every vertex in $N(v)\setminus N(u)$ must be adjacent to $w$. If $d_{N(u)}(v)\leq 2$, then there is a vertex $u' \in N(u)\setminus N(v)$, and a vertex $v' \in N(v) \setminus N(u)$, since $\delta(G) \geq 3$. Then $\{u,v\} \cup (N(u) \triangle N(v)) \subseteq N(w)$, which means $d(w) > d(u)$, contradicting the choice of $u$. Thus, $d_{N(u)}(v)\geq 3$ (i.e. $|L|\geq 3$). If there is a vertex $x \in N(u) \setminus N(v)$, then it must be adjacent to every vertex in $L$, since otherwise we obtain an induced $P_4$ as above. Then, $\{x,u,v\}$ together with three vertices in $L$ form a copy of $K_{3,3}$ in $G$, contradicting the planarity of $G$. Thus, $N(u) \subseteq N(v)$, and similarly $N(v) \subseteq N(u)$, that is, $N(u)=N(v)=L$. Recall that $\Delta(G)\geq 5$, and hence,  $l:=|L| \geq 5$.
   
   Let $L=\{x_1, x_2, \ldots, x_l\}$. Then in a plane drawing of $G$, the subgraph $G':=G[\{u,v\} \cup L]$ divides the plane into regions bounded by $ux_ivx_j$, for consecutive (when ordered in the obvious way) vertices of $L$. Any vertex $x_i\in L$ can only be adjacent to $x_{i-1}$ and $x_{i+1}$ in $L$, where $+$ is modulo $l$. 
   If a vertex $x_i \in L$ that has a neighbor $y\in V(G) \setminus (\{u,v\} \cup L)$, then $y$ must lie in a region with $x_i$ is on its boundary, that is, $y$ can only be adjacent to $x_i$ and $x_{i+1}$, or $x_i$ and $x_{i-1}$ in $L$. Since $|L|\geq 5$, then there is a vertex $x_j \in L$ which is neither adjacent to $x_i$ nor adjacent  to $y$. Hence, $x_jux_iy$ is an induced $P_4$ in $G$, a contradiction. Since $G$ is connected, we obtain that $L=V(G) \setminus \{u,v\}$.


   Since only consecutive vertices of $L$ can be adjacent to each other, $G[L]$ contains no cycle of length $l'<l$, then any $P_4$ in $G[L]$ would be an induced $P_4$ (which means $G[L]$ does not contain $C_l$ either, since $l\geq 5$). Thus, $G[L]$ does not contain any $P_4$, i.e. every component of $G[L]$ is a path on at most three vertices. Observe that a component of order one in $G[L]$ means a vertex of degree two in $G$, which is a contradiction. If $C$ is a component of $G[L]$, then if it has three vertices, it is incident to $8$ edges, and if it has two vertices, it is incident to $5$ edges, which means that in any case $C$ is incident to $\lfloor \frac{8}{3}|C|\rfloor$ edges.
   We can now delete $C$ and apply induction to obtain $$|E(G)|=|E(G\setminus C)|+ \lfloor \frac{8}{3}|C|\rfloor \leq \lfloor \frac{8}{3}((n-|V(C)|)-2) \rfloor+1+\lfloor \frac{8}{3}|C|\rfloor=\lfloor\frac{8}{3}(n-2) \rfloor+1,$$
   where the inequality is by the induction hypothesis for every $n\geq 9$. Note that if $n=8$, then $n-|V(C)|<7$, but still the inequality is true for $G\setminus C$. Since if $|V(C)|=3$ then $n-|V(C)|=n-3=5$, and $G\setminus C$ has at most $9=\lfloor\frac{8}{3}((n-3)-2)\rfloor+1$ edges. Also, if $|V(C)|=2$, then $n-|V(C)|=n-2=6$, but in this case $G\setminus C$ cannot be a triangulation on $6$ vertices. Since on $6$ vertices, the only triangulation without induced $P_4$ is $\overline{K}_2+C_4$, which cannot have any other vertex joined to it without containing an induced $P_4$. Thus, $G\setminus C$ would have at most $11=\lfloor\frac{8}{3}((n-2)-2)\rfloor+1$ edges. 
   
   


   \vskip3mm
   \textbf{Case 2.} $W = \emptyset$.

   In this case, $d(u)=n-1$. Choose a vertex $v\in V(G)\setminus \{u\}$ with maximum degree in $G':=G\setminus u$.
   If $d_{G'}(v)\leq 2$, then $e(G)\leq 2n-2< \lfloor\frac{8}{3}(n-2)\rfloor+1$, for every $n\geq 8$, a contradiction. Then, $d_{G'}(v) \geq 3$, that is, $v$ is adjacent to $u$ and has at least three common neighbors with $u$. Let $L:=N(u) \cap N(v)$, then $|L|\geq 3$. 
   

   Let $L:=\{x_1,x_2, \ldots, x_l\}$. Order the vertices in $L$ so that in a plane drawing of $G$, the subgraph $G[\{u,v\} \cup L]$ of $G$ separates the plane into regions that are bounded by $uvx_1$, $ux_1vx_2, \ldots,$ $ ux_ivx_{i+1}, \ldots, uvx_l$. 
   Then, only consecutive vertices in $L$ can be adjacent to each other, which means $G[L]$ does not contain any cycle. Thus, any $P_4$ in $G[L]$ would be an induced $P_4$, and hence, $G[L]$ does not contain a $P_4$. Therefore, the components of $G[L]$ are paths of at most three vertices. 
   
   Similarly to the previous case, we show that $L=V(G) \setminus \{u,v\}$.
   Let $y \in V(G)\setminus(\{v\} \cup N(v))$ be a vertex that is adjacent to a vertex $x_i\in L$.

   If $l\geq 4$, then none of $x_i$ and $y$ can be adjacent to $x_{i+2}$, and hence $yx_ivx_{i+2}$ is an induced $P_4$, a contradiction.

   Assume that $l=3$, i.e. $d_{G'}(v)=3$. Still, the above argument is true if $i=1$ or $3$. So assume that $y$ is adjacent to $x_2$. Then it is in the region bounded by $ux_1vx_2$ or in the region bounded by $ux_2vx_3$. without loss of generality, suppose that $y$ is in the region bounded by $ux_1vx_2$. If $y$ is also adjacent to $x_1$, then we are done as before, since $yx_1vx_3$ would be an induced $P_4$. So assume that $y$ is not adjacent to $x_1$. If $x_2$ is not adjacent to $x_1$, then $yx_2vx_1$ is an induced $P_4$. Similarly, if $x_2$ is not adjacent to $x_3$, then $yx_2vx_3$ is an induced $P_4$. Therefore, $x_2$ is adjacent to all of $y,v,x_1$ and $x_3$, which means that $d_{G'}(x_2)>d_{G'}(v)$, contradicting the choice of $v$.

   Now, similarly to the previous case, we can find a component $C$ in $G[L]$ which is incident to $\lfloor \frac{8}{3} |C|\rfloor$ edges and conclude the result by induction. 
   \end{proof}

\section{Graphs with no induced  $P_5$}

Here, we prove Theorem \ref{P_5}. First, we need some preliminary results. Let $G$ be a planar graph, define operations $D_4$ and $D_5$ on $G$ as follows.  If $u$ is a vertex of degree $4$ in $G$ such that $N(u)$ forms a cycle of length four, then $D_4$ is the operation of deleting $u$ and adding a diagonal edge for the $4$-face bounded by $N(u)$ after deleting $u$. If $u$ is a vertex of degree $5$ such that $N(u)$ forms a cycle of length five, then $D_5$ is the operation of deleting $u$ and adding the two possible diagonals for the $5$-face bounded by $N(u)$ after deleting $u$, see Figure \ref{D4D5}. Observe that these operations are simply the contraction of one of the edges incident to $u$.

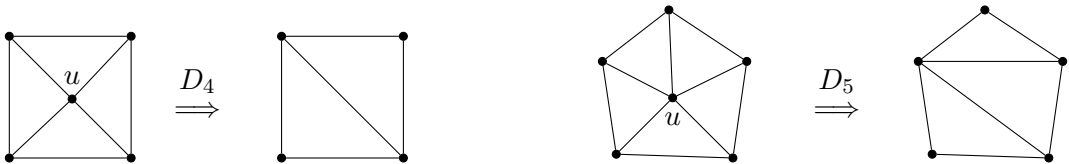
\begin{figure}[h]
        \centering
\begin{center}
    \begin{tikzpicture}
    \node[vertex] (a1) {};
    \node[vertex] (b1) [right=1.5cm of a1] {};
    \node[vertex] (c1) [below=1.5cm of b1] {};
    \node[vertex] (d1) [left=1.5cm of c1] {};
    \node[vertex, label=above:$u$] (u1) [below right=0.75cm and 0.75cm of a1] {};

    \draw (a1) -- (b1) -- (c1) -- (d1) -- (a1) -- (u1) -- (c1);
    \draw (b1) -- (u1) -- (d1);

   \node[] at (2.5,-1) {$\Longrightarrow$};
   \node[] at (2.5,-0.6) {$D_4$};

    \node[vertex] (a12) [right=3.5cm of a1]{};
    \node[vertex] (b12) [right=1.5cm of a12] {};
    \node[vertex] (c12) [below=1.5cm of b12] {};
    \node[vertex] (d12) [left=1.5cm of c12] {};

    \draw (a12) -- (b12) -- (c12) -- (d12) -- (a12) -- (c12);

    \node[vertex] (a21) [below right=0.25 and 7.8cm of a1]{};
    \node[vertex] (b21) [above right=0.6cm and 0.8cm of a21]{};
    \node[vertex] (c21) [right=1.8cm of a21]{};
    \node[vertex] (d21) [below left=1.2cm and 0.1cm of c21]{};
    \node[vertex] (e21) [below right=1.15cm and 0.1cm of a21]{};
    \node[vertex, label=below:$u$] (u21) [below right= 0.4cm and 0.85cm of a21] {};

    \draw (a21) -- (b21) -- (c21) -- (d21) --(e21) -- (a21) -- (u21) -- (c21);
    \draw (b21) -- (u21) -- (d21);
    \draw (e21) -- (u21);

    \node[] at (11,-1) {$\Longrightarrow$};
   \node[] at (11,-0.6) {$D_5$};

   \node[vertex] (a22) [below right=0.25cm and 12cm of a1]{};
    \node[vertex] (b22) [above right=0.6cm and 0.8cm of a22]{};
    \node[vertex] (c22) [right=1.8cm of a22]{};
    \node[vertex] (d22) [below left=1.2cm and 0.1cm of c22]{};
    \node[vertex] (e22) [below right=1.15cm and 0.1cm of a22]{};

    \draw (a22) -- (b22) -- (c22) -- (d22) --(e22) -- (a22);
    \draw (c22) -- (a22) -- (d22);

\end{tikzpicture}
 \end{center}   
        \caption{Illustrating the operations $D_4$ and $D_5$}
        \label{D4D5}
    \end{figure}

\begin{lemma}\label{DkeepsP5ind-freeness}
    Let $G'$ be a planar graph obtained from $G$ by applying the operation $D_4$ or $D_5$. If $G$ is $P_5^\ind$-free, then so is $G'$.
\end{lemma}

\begin{proof}
    Assume that $G'$ is obtained by applying $D_4$, and contains an induced $P_5$. Let $u$ be the vertex of degree $4$, and $N(u):=\{x,y,z,w\}$. Let the four cycle formed by $N(u)$ be $xyzw$ and $xz$ be the added edge in $G'$. Let $P$ be an induced copy of $P_5$ in $G'$. Since $G$ is $P_5^\ind$-free, then $P$ must contain $xz$. Thus, $P$ cannot contain any of $y$ and $w$, since otherwise it would not be induced. Thus, replacing $xz$ by $xuz$ gives an induced $P_6$ in $G$, a contradiction.

    Now, assume $G'$ is obtained by applying $D_5$. Let $xyzwv$ be the $5$-cycle formed by the five neighbors of $u$. without loss of generality, let $xz$ and $xw$ be the two diagonal edges that are added after deleting $u$. Let $P$ be an induced copy of $P_5$ in $G'$. Then, it must contain exactly one of $xz$ and $xw$, as otherwise it would not be an induced copy. Likewise, $P$ cannot use any of $y$ and $w$. If it does not contain $v$ either, then similarly to the previous case, we obtain a contradiction. If $P$ contains $v$, then it must contain $vxz$ as a segment, then the other vertices of $P$ cannot be adjacent to $x$. Then, we can replace $vxz$ by $vuz$ in $G$ and obtain an induced $P_5$, a contradiction.
\end{proof}

\begin{figure}[h]
        \centering
\begin{center}
    \begin{tikzpicture}
    \node[vertex] (a1) {};
    \node[vertex] (b1) [right=4cm of a1] {};
\node [vertex] (c1) [above right=0.7cm and 2cm of a1] {};
\node [vertex] (d1) [above right=1.3cm and 2cm of a1] {};
\node [vertex] (e1) [above right=1.8cm and 2cm of a1] {};
\node [vertex] (f1) [above right=3cm and 2cm of a1] {};

\node [vertex] (g1) [above right=2cm and 1.65cm of a1] {};
\node [vertex] (h1) [above right=2cm and 2.35cm of a1] {};

\draw[line width=0.5pt] (a1) -- (b1) -- (c1)-- (a1)-- (d1) -- (b1) -- (e1) -- (a1) -- (f1) -- (b1);
\draw[line width=0.5pt] (c1) -- (d1)-- (e1) -- (f1);
\draw[line width=0.5pt] (a1) -- (g1)-- (e1) -- (h1) -- (b1);
\draw[line width=0.5pt] (g1) -- (f1)-- (h1) -- (b1);

\node[] at (2.2,-0.7) {$A$};


\node[vertex] (a2) [right=5cm of a1] {};
\node[vertex] (b2) [right=4cm of a2] {};
\node [vertex] (c2) [above right=0.6cm and 2cm of a2] {};
\node [vertex] (d2) [above right=1.1cm and 2cm of a2] {};
\node [vertex] (e2) [above right=1.9cm and 2cm of a2] {};
\node [vertex] (f2) [above right=3cm and 2cm of a2] {};

\node [vertex] (g2) [above right=2cm and 1.65cm of a2] {};
\node [vertex] (h2) [above right=1.25cm and 2.35cm of a2] {};

\draw[line width=0.5pt] (a2) -- (b2) -- (c2)-- (a2)-- (d2) -- (b2) -- (e2) -- (a2) -- (f2) -- (b2);
\draw[line width=0.5pt] (c2) -- (d2)-- (e2) -- (f2);
\draw[line width=0.5pt] (a2) -- (g2)-- (e2) -- (h2) -- (b2);
\draw[line width=0.5pt] (g2) -- (f2);
\draw[line width=0.5pt] (h2) -- (d2);

\node[] at (7.2,-0.7) {$B$};


\node[vertex] (a3) [right=10cm of a1] {};
\node[vertex] (b3) [right=4cm of a3] {};
\node [vertex] (c3) [above right=0.6cm and 2cm of a3] {};
\node [vertex] (d3) [above right=1.4cm and 2cm of a3] {};
\node [vertex] (e3) [above right=1.9cm and 2cm of a3] {};
\node [vertex] (f3) [above right=3cm and 2cm of a3] {};

\node [vertex] (g3) [above right=2cm and 1.65cm of a3] {};
\node [vertex] (h3) [above right=0.8cm and 2.35cm of a3] {};

\draw[line width=0.5pt] (a3) -- (b3) -- (c3)-- (a3)-- (d3) -- (b3) -- (e3) -- (a3) -- (f3) -- (b3);
\draw[line width=0.5pt] (c3) -- (d3)-- (e3) -- (f3);
\draw[line width=0.5pt] (a3) -- (g3)-- (e3) -- (d3) -- (h3) -- (b3);
\draw[line width=0.5pt] (g3) -- (f3);
\draw[line width=0.5pt] (h3) -- (c3);

\node[] at (12.2,-0.7) {$C$};


\node[vertex] (a4) [below=5cm of a1] {};
\node[vertex] (b4) [right=4cm of a4] {};
\node [vertex] (c4) [above right=0.6cm and 2cm of a4] {};
\node [vertex] (d4) [above right=1.1cm and 2cm of a4] {};
\node [vertex] (e4) [above right=1.9cm and 2cm of a4] {};
\node [vertex] (f4) [above right=2.5cm and 2cm of a4] {};

\node [vertex] (g4) [above right=1.25cm and 1.65cm of a4] {};
\node [vertex] (h4) [above right=1.25cm and 2.35cm of a4] {};

\draw[line width=0.5pt] (a4) -- (b4) -- (c4)-- (a4)-- (d4) -- (b4) -- (e4) -- (a4) -- (f4) -- (b4);
\draw[line width=0.5pt] (c4) -- (d4)-- (e4) -- (f4);
\draw[line width=0.5pt] (a4) -- (g4)-- (e4) -- (h4) -- (b4);
\draw[line width=0.5pt] (g4) -- (d4) -- (h4) -- (b4);

\node[] at (2.2,-5.7) {$D$};

\node[vertex] (a5) [right=5cm of a4] {};
\node[vertex] (b5) [right=4cm of a5] {};
\node [vertex] (c5) [above right=0.6cm and 2cm of a5] {};
\node [vertex] (d5) [above right=0.9cm and 2.5cm of a5] {};
\node [vertex] (e5) [above right=1.5cm and 2.5cm of a5] {};
\node [vertex] (f5) [above right=3cm and 2cm of a5] {};

\node [vertex] (g5) [above right=1.9cm and 2cm of a5] {};
\node [vertex] (h5) [above right=1.25cm and 1.35cm of a5] {};

\draw[line width=0.5pt] (a5) -- (b5) -- (c5)-- (a5)-- (f5) -- (b5);
\draw[line width=0.5pt] (c5) -- (d5)-- (e5) -- (g5) -- (h5) -- (c5);
\draw[line width=0.5pt] (a5) -- (h5)-- (f5) -- (g5) -- (c5);
\draw[line width=0.5pt] (f5) -- (e5) -- (b5) -- (d5) -- (g5);

\node[] at (7.2,-5.7) {$E$};

\node[vertex] (a6) [right=10cm of a4] {};
\node[vertex] (b6) [right=4cm of a6] {};
\node [vertex] (c6) [above right=0.5cm and 2cm of a6] {};
\node [vertex] (d6) [above right=1cm and 2cm of a6] {};
\node [vertex] (e6) [above right=2cm and 2cm of a6] {};
\node [vertex] (f6) [above right=3cm and 2cm of a6] {};

\node [vertex] (g6) [above right=1.5cm and 2.55cm of a6] {};
\node [vertex] (h6) [above right=1.5cm and 1.45cm of a6] {};

\draw[line width=0.5pt] (a6) -- (b6) -- (c6)-- (a6)-- (d6) -- (b6) -- (f6) -- (a6) -- (h6) -- (f6) -- (g6) -- (b6);
\draw[line width=0.5pt] (h6) -- (g6);
\draw[line width=0.5pt] (f6) -- (e6) -- (h6) -- (d6) -- (g6) -- (e6);
\draw[line width=0.5pt] (c6) -- (d6);

\node[] at (12.2,-5.7) {$F$};

\end{tikzpicture}
 \end{center}   
        \caption{The $8$-vertex triangulations with no induced $P_5$.}
        \label{fig:8trnoP5}
    \end{figure}
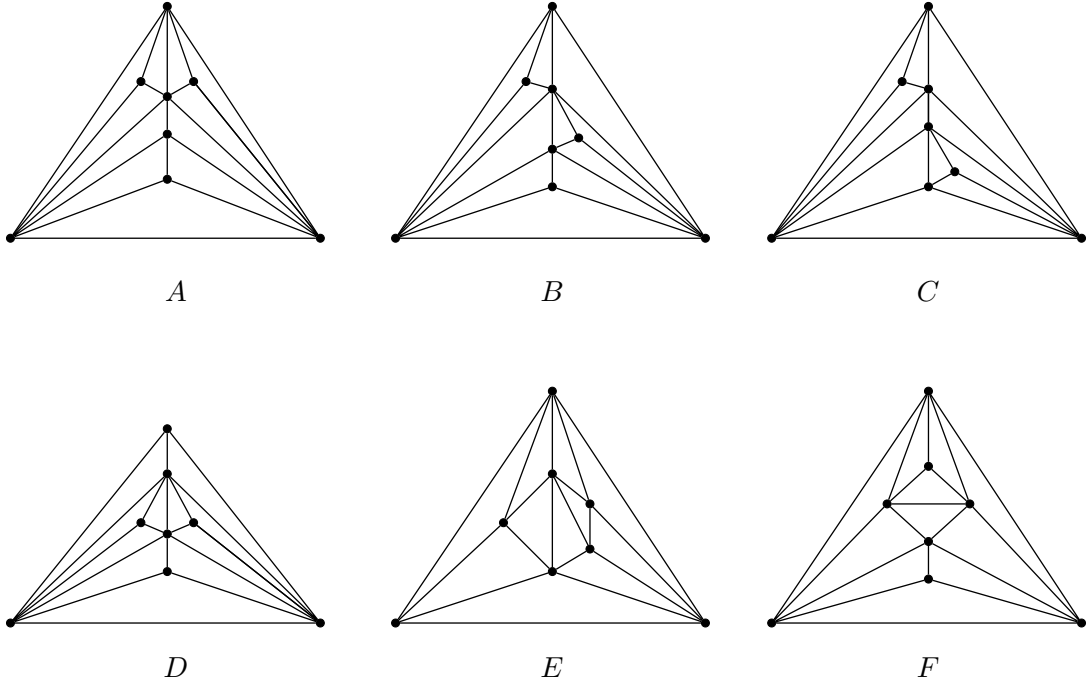

\begin{lemma}\label{>9tr>indP5}
    Any triangulation on at least $9$ vertices contains an induced $P_5$.
\end{lemma}

\begin{proof}
    Let $T_k$ be a triangulation on $k\geq 9$ vertices. The proof is by induction on $k$. Let $k=9$. Assume $T_9$ is a triangulation that does not contain an induced $P_5$. If $T_9$ has a vertex $u$ of degree $3$, then $T':=T_9\setminus \{u\}$ is an $8$-vertex triangulation. There are 15 non-isomorphic triangulations on $8$ vertices, and its not difficult to check that there are only $6$ of them that does not contain an induced $P_5$, which are $A, B, C, D, E$ and $F$, given in Figure \ref{fig:8trnoP5}. Thus, $T' \in \{A, B, C, D, E, F\}$. But it is easy to check that adding any vertex of degree $3$ to any of these, the resulting graph will contain an induced $P_5$, a contradiction. Thus, we may assume that $T_9$ has minimum degree $4$. Note that any triangulation of at most 11 vertices has vertex of degree at most 4. Let $u \in V(T_9)$ with $d(u)=4$. Obviously $N(u)$ forms a cycle in $T_9$. Then, apply the operation $D_4$ to obtain $T'$. Clearly, $T'$ is a triangulation on $8$ vertices, and by Lemma \ref{DkeepsP5ind-freeness}, it is $P_5^\ind$-free. Thus, $T'\in \{A, B, C, D, E, F\}$. Note that by definition of the operation $D_4$, $T_9$ is obtained from $T'$ by subdividing an edge $e$ once and joining the new vertex to the other two vertices on the boundary of the two $3$-faces on each side of $e$. Again, it is easy to check that in any of $A, B, C, D, E$ and $F$,  every edge belongs to an induced $P_4$, and hence subdividing it gives an induced $P_5$ in the resulted graph. Thus, $T_9$ contains an induced $P_5$, a contradiction.

    Now, assume that $k>9$, and any triangulation on $k-1$ vertices contains an induced $P_5$. If $T_k$ has a vertex of degree $3$, then deleting it results in a triangulation on $k-1$ vertices, which, by the induction hypothesis, contains an induced $P_5$. Thus, $T_k$ itself contains an induced $P_5$. If $T_k$ has a vertex of degree $4$, apply the operation $D_4$, and if it has a vertex of degree $5$, apply the operation $D_5$. Then the resulted graph is a triangulation on $k-1$ vertices, which contains an induced $P_5$ by the induction hypothesis. Then, by Lemma \ref{DkeepsP5ind-freeness}, $T_k$ contains an induced $P_5$, too.
\end{proof}

In what follows, the following simple observations are repeatedly used. Let $G$ be a planar graph and $X=\{x_1, \ldots, x_k\} \subseteq V(G)$ be a cut-set in $G$. For a component $H$ of $G \setminus X$, and a vertex $v\in X$, we say $P_l$ is an induced path from $v$ into $H$, if $v$ is the starting vertex of $P_l$ and apart from $v$ every other vertex of $P_l$ is in $H$, i.e. $P_l$ is an induced path in $G[\{v\} \cup V(H)]$ with $v$ a starting vertex.

\begin{obs}\label{obs1} 
If $G$ is $k$ connected then each vertex of $X$ has at least one neighbor in any component of $G\setminus X$.
\end{obs}

    Indeed, if a vertex $v\in X$ has no neighbors in some component of $G \setminus X$, then $X\setminus \{v\}$ would be a cut-set in $G$.

    \begin{obs}\label{obs2}
        If $G$ is $k$-connected and $k \geq 3$, then
    $G\setminus X$ has exactly two components. 
    \end{obs} 
    
    Indeed, in each component of $G\setminus X$, contract all the edges repeatedly until it shrinks to one vertex, then by the first observation, this vertex is adjacent to every vertex in $X$. If there were at least three  components, then after the contractions we would have a copy of $K_{3,3}$, which contradicts the fact that contracting edges preserves planarity.

    \begin{obs}\label{obs3}
        If $k\geq 3$, then $x_1, \ldots, x_k$ have at most one common neighbor in each component of $G \setminus X$.
    \end{obs}

    Otherwise, if they have two common neighbors in a component, say $H_1$. Contract the other component into one vertex, this would give the third common neighbor to give a copy of $K_{3,3}$.

    \begin{obs}\label{obs4} 
    If $G$ is $P_5^\ind$-free, and $v \in X$, then there is no induced $P_4$ from $v$ into any component of $G \setminus X$.
    \end{obs}
    
Since, by Observation 1, we can take a neighbor of $x$ in another component and get an induced $P_5$ in $G$.

    \begin{obs}\label{obs5}
        If $G$ is $P_5^\ind$-free, and $v \in X$, then there is no induced $P_3$ from $v$ into two components of $G \setminus X$.
    \end{obs} 

    Obviously, if so, we have an induced $P_5$ in $G$.

    \begin{obs}\label{obs6} Let $v\in X$, and $u \in V(H)$ for some component of $G\setminus X$ such that $uv \notin E(G)$. Then, there is an induced $P_3$ from $v$ into $H$.
    \end{obs}
 
    Since $v$ has a neighbor in $H$, and $H$ is connected, take a path in $H$ that joins the neighbor of $v$ to $u$. Then there is a path from $v$ into $H$ with $u$ as the end vertex. The shortest such path cannot have length one since $u$ is not adjacent to $v$, and cannot have $4$ vertices by Observation 4.

\begin{lemma}\label{3connseptriangle}
    Let $G$ be a $3$-connected $P_5^\ind$-free planar graph with $X=\{x_1,x_2,x_3\}$ a cut-set in $G$. Let $G'$ be obtained from $G$ by adding all the missing edges in $X$. 
    Then $G'$ is also a $P_5^\ind$-free planar graph.
\end{lemma}

\begin{proof}
    Let $X=\{x_1,x_2,x_3\}$ be a cut-set of vertices. Let $H_1$ and $H_2$ be the components of $G\setminus X$. Let $G_1:=G[X \cup V(H_1)]$ and $G_2:=G[X \cup V(H_2)]$. 
    If $x_1x_2x_3$ is a triangle in $G$, then we are done, so we may assume that some edges among the vertices of $X$ are missing. Since $H_2=G \setminus G_1$ and, by Observation \ref{obs1}, every vertex of $X$ has a neighbor in $H_2$ then in a plane drawing of $G$, we have $H_2$ lying in a face of $G_1$ that has $x_1, x_2$ and $x_3$ on its boundary. As there are no edges from the vertices of $H_2$ to the vertices in $H_1$, adding the missing edges between the vertices of $X$ does not violate the planarity of $G$, showing that $G'$ is planar. 
    
    Suppose that $G'$ contains an induced $P_5$. Denote such a copy by $P$. Since $G$ is $P_5^\ind$-free, then $P$ must contain one of the newly added edges, say $x_1x_2\in E(P)$ (that means $x_1x_2$ was missing in $G$). Then, as $x_1x_2x_3$ is a triangle in $G'$, and $P$ is an induced path, $x_3 \notin V(P)$. 
    
    Assume that $V(P)$ lies entirely in one of $G_1$ or $G_2$. Without loss of generality, suppose  it is in $G_1$. If $x_1x_2$ is the last edge of $P$ with, say $x_2$ as the last vertex, then we can replace the edge $x_1x_2$ by $x_1u$ for some neighbor $u \in H_2$ of $x_1$, to get an induced $P_5$ in $G$, a contradiction. If $x_1x_2$ is not the last edge of $P$, then we can replace it by a shortest path joining $x_1$ and $x_2$ in $G_2$ to get an induced path of length at least $5$ in $G$, which contains an induced $P_5$, a contradiction.

    Thus, we may assume that $P$ intersects both $V(H_1)$ and $V(H_2)$. Then $P$ contains $x_1x_2$, an induced path $P_3$ starting from $x_1$ or $x_2$ into one of $H_1$ or $H_2$, and a neighbor of the other vertex of $\{x_1, x_2\}$ in the other component. Without loss of generality, let $P:=uvx_1x_2w$, where $u, v \in V(H_1)$ and $w \in V(H_2)$. Since $P$ is induced, $w$ is not adjacent to $x_1$, 
    so the shortest path that joins $w$ to $x_1$ in $G[H_2 \cup \{x_1\}]$ has length at least two. Then using this path and $x_1vu$ in $G_1$, we can get an induced $P_5$ in $G$, a contradiction.
\end{proof}




\begin{proposition}\label{3conP_5ind}
    Every $3$-connected planar graph on at least 13 vertices contains an induced $P_5$.
\end{proposition}

\begin{proof}
    First, assume that $G$ has at least 11 vertices and has connectivity $3$. Let $X:=\{x_1,x_2,x_3\}$ be a cut-set in $G$. By Lemma \ref{3connseptriangle}, we may assume that $X$ is a separating triangle. Let $H_1$ and $H_2$ be the components of $G \setminus X$, and for $j=1,2$, let $G_j:=G[X \cup V(H_j)]$. 
    
    Towards a contradiction, assume that $G$ does not contain an induced $P_5$. 


    \begin{claim}\label{Clm2indP3}
        For any $i=1,2$, if $|V(H_i)|\geq 3$, then there are at least two vertices of $X$ from which there is an induced $P_3$ into $H_i$.
    \end{claim}

    \begin{proof}
         Assume that $H_1$ contains at least $3$ vertices. By Observation \ref{obs6}, it suffices to show that there are at least two vertices in $X$ that have non-neighbors in $H_1$. 

        By Observation \ref{obs3}, $x_1, x_2$ and $x_3$ can have at most one common neighbor in $H_1$, and $H_1$ has at least three vertices, then we immediately obtain that at least one vertex of $X$, say $x_3$, has a non-neighbor in $H_1$. Suppose that $x_1$ and $x_2$ are adjacent to every vertex of $H_1$. Since every vertex of $X$ has a neighbor in $H_1$, let $w$ be the neighbor of $x_3$. Thus, $w$ is a common neighbor of all the three vertices of $X$. This divides the region that contains $H_1$ into three regions $R_i$ bounded by triangles $x_ix_{i+1}w$, for $i=1,2,3$, where $+$ it taken modulo $3$. Note that for each $i=1,2,3$, the vertices in $R_i$ are separated from $x_{i+2}$, and hence cannot be adjacent to $x_{i+2}$. Thus, the rest of the, at least two, other vertices of $H_1$ (apart from $w$) must lie inside the region $R_1$. 
        
        Thus, $\{x_1,x_2,w\}$ is again a separating triangle in $G$ with at least two vertices in the component in its interior region $R_1$. 
        If there is a vertex $u$ inside $R_1$ that is not adjacent to $w$, then by Observation \ref{obs6}, there is an induced $P_3$ from $v$ to $u$, and hence an induced $P_4$ from $x_3$ to $u$, contradicting Observation \ref{obs4}. Thus, every vertex in $R_1$ must be adjacent to $w$.  Hence, there are at least two vertices inside $R_1$, and all of them are adjacent all the three of $x_1, x_2$ and $w$, contradicting Observation \ref{obs3}.
        
        Thus, at least one of $x_1$ and $x_2$ also has a non-neighbor in $H_1$, and we are done. 
    \end{proof}
    Now, let us return to the proof of the Lemma. If each of $H_1$ and $H_2$ has at least $3$ vertices, then by the above claim, and applying the pigeonhole principle, there is a vertex in $X$ from which we can start an induced $P_3$ into $H_1$, and an induced $P_3$ into $H_2$, which contradicts Observation \ref{obs5}.

    Suppose that $|V(H_1)| \leq 2$. Then, $|V(H_2)| \geq 11-3-2=6$. For each $x_i\in X$, let $N_i:=N(x_i) \cap V(H_2)$. 

    \smallskip
    \textbf{CASE 1.} For every $i\in \{1,2,3\}$, $N_i \cap N_{i+1}=\emptyset$. 
    
    Then, every vertex in $N_i$ is at distance two from each of $x_{i+1}$ and $x_{i+2}$. In particular, from each vertex of $X$, we can start an induced $P_3$ into $H_2$, which means each vertex of $H_1$ must be adjacent to every vertex of $X$ (otherwise we can find an induced $P_5$ in $G$), and hence $|V(H_1)|=1$. Thus, in this case, $|V(H_2)|\geq 7$.

    If a vertex $y\in N_i$ is not adjacent to a vertex $z \in N_{i+1}$, then there is an induced path of length at least two that joins them in $H_2$. the internal vertices of this path are either in $N_i$ or $N_{i+1}$, but not both. Thus, there is an induced $P_4$ starting from either $x_i$ or $x_{i+1}$ into $H_2$, a contradiction. Thus, for each $i=1,2,3$, $N_i$ and $N_{i+1}$ induce a complete bipartite graph. If for any $i$, $|N_i|\geq 2$, then $G$ would contain a subdivision of $K_{3,3}$ with one part consisting of two vertices in $N_i$ and the vertex in $H_1$, and the second part consisting of a vertex in $N_{i+1}$, a vertex in $N_{i+2}$ and $x_i$, contradicting the planarity of $G$. Thus, $|N_i|=1$, for every $i=1,2,3$. Thus, $W:=V(H_2) \setminus \cup_{i=1}^3N_i \geq 7-3=4$.


    Let $N_i:=\{y_i\}$, for each $i=1,2,3$. If a vertex $w \in W$ is not adjacent to $y_i$, for some $i$, then it is at distance at least two from $y_i$ in $H_2$, and hence $x_i$ can be the starting vertex of an induced $P_4$ into $H_2$, a contradiction. Hence, every vertex of $W$ is adjacent to every vertex in $\cup_{i=1}^3N_i$. 
    Thus, we would find a subdivision of $K_{3,3}$ in $G$, with $\{y_1,y_2,y_3\}$ as one part, and the other part consists of three vertices in $W$,
    a contradiction.

    \smallskip

\textbf{CASE 2.} For some $i \in \{1,2,3\}$, $N_i \cap N_{i+1}\neq \emptyset$.  

    \textbf{Subcase 2.1.} There is a vertex $w$ that is a common neighbor to all vertices in $X$. 
    
    Then the region containing $H_2$ is divided into three regions $R_i$, bounded by $x_ix_{i+1}w$, for each $i=1,2,3$. Note that if a region $R_i$ is not empty, then its boundary, $x_i x_{i+1} w$, forms a separating triangle in $G$, and hence a cut-set of size three. Since $G$ is 3-connected, each of the boundary vertices must have at least one neighbor inside $R_i$. In particular, $Y:=\{x_i,x_{i+1}, w\}$ is a cut-set and $G\setminus Y$ has a component $F_1$ on the vertices inside $R_i$ and another component $F_2$ with its vertex set consisting of $V(H_1) \cup \{x_{i+2}\}$ and the vertices inside the other two regions $R_{i+1}$ and $R_{i+2}$.
    
    Thus, If a region $R_i$ contains at least $3$ vertices, then the other two regions must be empty and $|V(H_1)|=1$, otherwise $F_2$ would also have at least three vertices, and we can apply Claim \ref{Clm2indP3} as above, to get a contradiction. Without loss of generality, assume that $R_1$ contains at least three vertices. Then, $R_2$ and $R_3$ are empty and $|V(H_1)|=1$, which means $R_1$ contains at least $6$ vertices. Let $V(H_1)=\{u\}$, then $Y=\{x_1, x_2,w\}$, $V(F_2)=\{u,x_3\}$, and $V(F_1)$ is the set of vertices inside $R_1$.

    Since $wx_3u$ is an induced $P_3$ starting from $w$ into $F_2$, $w$ must be adjacent to every vertex of $F_1$ (i.e. inside $R_1$). Thus, if $x_1$ and $x_2$ have a common neighbor inside $R_1$, then it would be a common neighbor of every vertex of $Y$. Let $w'$ be such a common neighbor. This further divides $R_1$ into three regions bounded by triangles $x_1x_2w', x_1ww'$ and $x_2ww'$. Since any vertex inside the region bounded by $x_1x_2w'$ would be separated from $w$, this region must be empty. Let $Q_i$ be the regions bounded by $x_iww'$, for each $i=1,2$. Since there are at least $5$ vertices inside $R_1$ (apart from $w'$), one of $Q_1$ and $Q_2$ contains at least three vertices, say $Q_1$ is such. Then, we again have $Y':=\{x_1,w,w'\}$ is a cut-set of size three in $G$, which is a separating triangle that contains at least three vertices in its interior region and at least three vertices in its exterior (note that each of $x_2, x_3$ and $u$ are in the exterior region). Thus, we can again apply Claim \ref{Clm2indP3} and obtain a contradiction.

    Suppose that for each $i=1,2,3$, the region $R_i$ contains at most two vertices. Since $|V(H_2)|\geq 6$, then for each $i=1,2,3$, the region $R_i$ is not empty, and at least two of them contain two vertices.

    Let $R_1$ be a region that contains two vertices. Since $x_1,x_2$ and $w$ can have at most one common neighbor inside $R_1$, and every vertex inside $R_1$ is adjacent to $w$ (similarly as above), then one of $x_1$ or $x_2$ has a non-neighbor in $R_1$. 
    Let $x_1$ be the vertex with a non-neighbor inside $R_1$. Then, we can start an induced $P_3$ from $x_1$ into $R_1$ (without using $x_2$ and $w$), say $x_1z_1z_2$ is such an induced $P_3$. Let $z$ be a neighbor of $x_3$ inside the region  $R_2$, which exists by Observation \ref{obs1}. Then, $zx_3x_1z_1z_2$ is an induced $P_5$ in $G$, a contradiction.

\smallskip
   \textbf{Subcase 2.2.} No vertex of $H_2$ is a common neighbor of $x_1, x_2$ and $x_3$.   

   Let $a \in N_1 \cap N_2$. Then $a$ is not adjacent to $x_3$. Then the shortest path from $a$ to $x_3$ in $G[V(H_2) \cup \{x_3\}]$ has length 2. 
   Let $x_3ba$ be such an induced path from $x_3$ to $a$.

   Now, $b$ can be adjacent to at most one of $x_1$ and $x_2$. Assume that $b$ is not adjacent to $x_2$. Then, similarly $x_2ab$ is an induced $P_3$. Consider the region $Q$ bounded by the triangle $x_1x_2a$. If $Q$ is not empty, then $x_1x_2a$ is a separating triangle, and hence, by Observation \ref{obs1}, $a$ must have a neighbor $c$ inside $Q$. Then, $cabx_3w$, where $w$ is a neighbor of $x_3$ in $H_1$, is an induced $P_5$ in $G$, a contradiction. Thus, $Q$ must be empty. In particular, $N_1 \cap N_2=\{a\}$, since if there are any other vertices in $N_1 \cap N_2$, they must be in the region $Q$, otherwise they would be separated from $x_1$ or from $x_2$. 

   If $b$ is adjacent to $x_1$, then similarly (this time there would be an induced $P_4$ from $x_2$ into $H_2$) the region bounded by $x_1bx_3$ must be empty, and $N_1 \cap N_3=\{b\}$. 

   Let $Q_1$ and $Q_2$ be the regions bounded by $x_1ab$ and $x_2abx_3$, respectively. If $Q_1$ is not empty, then $x_1ab$ is a separating triangle, which has more than three vertices in its exterior region (evident by $x_2, x_3$ and the vertices in $H_1$). So, if there are at least three vertices inside $Q_1$, then we can apply Claim \ref{Clm2indP3} and obtain a contradiction. Thus, $Q_1$ must contain at most two vertices, and hence $Q_2$ is not empty. But either $a$ or $b$ must have neighbors in $Q_2$, otherwise $\{x_2, x_3\}$ would be a cut-set of size two in $G$, a contradiction. Thus, for each $i=1,2,3$, there is an induced $P_3$ starting from $x_i$ into $H_2$, and hence every vertex in $H_2$ is adjacent to every $x_i$, which implies that $|V(H_1)|=1$. Thus, $|V(H_2)| \geq 7$. 
   
   Note that every vertex in $Q_1$ must be adjacent to both $a$ and $b$, otherwise, we can have an induced $P_4$ starting from $x_2$ (or $x_3$) into $H_2$, a contradiction. If there are two vertices in $Q_1$, then $x_1$ has a non-neighbor in $Q_1$ (otherwise $x_1, a$ and $b$ would have two common neighbors inside $Q_1$, contradicting Observation \ref{obs3}, and hence there is an  induced $P_3$ from $x_1$ into the vertices inside $Q_1$, say $x_1a_1a_2$. Since $|V(H_2)| \geq 7$, then $Q_2$ contains at least $3$ vertices. If no vertex in $Q_2$ is adjacent to any of $x_2$ and $x_3$, then $\{a,b\}$ would be a cut-set of size two in $G$, a contradiction. Thus, we may assume that $x_3$ has a neighbor, say $v$, inside $Q_2$. Then, $vx_3x_1a_1a_2$ is an induced $P_5$ in $G$, a contradiction. Thus, $Q_1$ contains at most one vertex. 

   Hence, $Q_2$ must contain at least $4$ vertices. 
   If a vertex $v$ inside $Q_2$ is not adjacent to any of $a$ and $b$, then there must be a path using only vertices inside $Q_2$ that joins $v$ to $a$ or to $b$. Since otherwise $\{x_2,x_3\}$ would be a cut-set of size two in $G$, a contradiction. Let $P$ be the shortest such path joining $v$ to $a$, then as $v$ is not adjacent to $a$, $P$ has length at least two. Thus, using $x_1a$ and this path we get an induced $P_4$ starting from $x_1$ into $H_2$, a contradiction. Therefore, every vertex inside $Q_2$ is adjacent to $a$ or  $b$. Let $A=N(a) \setminus N(b)$, $B=N(b)\setminus N(a)$ and $C=N(a) \cap N(b)$ inside $Q_2$. By the previous sentence, $A \cup B \cup C$ are all the vertices inside $Q_2$. If a vertex $v \in A$ is not adjacent to $x_3$, then $x_3bav$ is an induced $P_4$ starting from $x_3$ into $H_2$, a contradiction. So, every vertex of $A$ is adjacent to $x_3$. Similarly, every vertex of $B$ is adjacent to $x_2$. We claim that at least one of $A$ and $B$ must be empty. If $v\in A$, then since every vertex in $B$ and $C$ are adjacent to $b$, they must be inside the region bounded by $avx_3b$, as otherwise they would be separated from $b$. But then they are separated from $x_2$, and hence $B$ must be empty. Without loss of generality, we may assume that $A$ is not empty. Now, if $v \in A$, then it is not adjacent to $b$, and hence there is no path joining $v$ to $b$ using only vertices inside $Q_2$, since otherwise we would get an induced $P_4$ from $x_1$ into $H_2$. Thus, there must be a path joining $v$ to $x_2$, otherwise $\{a,x_3\}$ would be a cut-set in $G$. Choose $v \in A$ such that the distance from $v$ to $x_2$ is the shortest. If $v$ is not adjacent to $x_2$, then this distance it at least two. However, the internal vertices of the shortest path joining $v$ to $x_2$ must be adjacent to $a$, otherwise we get an induced $P_3$ from $a$ into $Q_2$, together with $x_1a$ we get an induced $P_4$ from $x_1$ into $H_2$, a contradiction. The vertices on the shortest path joining $v$ to $x_2$ are in the region bounded by $x_2avx_3$, so separated from $b$, and hence they are in $A$ with a shorter distance to $x_2$, contradicting the choice of $v$. Thus, $v$ is a common neighbor of each of $a, x_2$ and $x_3$. Then, all other (at least $3$) vertices inside $Q_2$ must lie in the region $Q_3$ bounded by $avx_3b$, otherwise they would be either separated from $x_3$ or from $b$, and hence in none of $A$ or $C$, a contradiction. Consider a vertex $u$ inside $Q_3$. Since $v \in A$, it is not adjacent to $b$, and hence $u$ cannot be adjacent to both $b$ and $v$, otherwise $x_1buv$ would be an induced $P_4$ from $x_1$ into $H_2$. Also, if $u$ is not adjacent to any of $b$ and $v$, then there is either an induced $P_3$ from $b$ to $u$ or an induced $P_3$ from $v$ to $u$, otherwise $\{a,x_3\}$ would be a cut-set, and hence there is an induced $P_4$ either from $x_1$ or from $x_2$ into $H_2$, a contradiction. Thus, $u$ is either adjacent to $b$ or adjacent to $v$, but not both. If $u$ is adjacent to $b$,  then, $vx_2x_1bu$ is an induced $P_5$ in $G$, and if $u$ is adjacent to $v$, then $uvx_2x_1b$ is an induced $P_5$ in $G$, a contradiction. 
   
   Thus, both $A$ and $B$ must be empty. This implies that every vertex of $Q_2$ are in $C$, the common neighbors of both $a$ and $b$. Recall that at least one of $x_2$ and $x_3$ must have a neighbor inside $Q_2$, otherwise $\{a,b\}$ would be a cut. Assume that $x_2$ has some neighbors inside $Q_2$. If it has at least two neighbors, then we would have two common neighbors of $a, b$ and $x_2$ inside $Q_2$, together with $x_1$, they form a copy of $K_{3,3}$ in $G$, a contradiction. Let $v$ be the only neighbor of $x_2$ inside $Q_2$. Since every vertex inside $Q_2$ is adjacent to both $a$ and $b$, they must lie in the region bounded by $abv$. As there are at least three vertices left, then we have that $abv$ is a separating triangle with at least three vertices both in its interior and in its exterior regions, and hence applying Claim \ref{Clm2indP3}, we obtain a contradiction.

Finally, assume that $b$ is not adjacent to $x_1$ (that is, it is not adjacent to any of $x_1$ and $x_2$). Then, again from each of $x_1, x_2$ and $x_3$ there is an induced $P_3$ into $H_2$, then $|V(H_1)|=1$, and hence $|V(H_2)|\geq 7$. Every vertex of $H_2$ apart from $a$ and $b$ is inside one of the regions $S_1$ bounded by $x_1abx_3$ or $S_2$ bounded by $x_2abx_3$.

If every path, with internal vertices in $H_2$, that joins $b$ to $x_1$ and $x_2$ uses $a$, then $\{a,x_3\}$ would be a cut-set of size two in $G$, a contradiction. Thus, there is a path joining $b$ to $x_1$ or to $x_2$ with internal vertices in $H_2$ and without using $a$. Without loss of generality, assume that there is such a path joining $b$ to $x_1$. Then, $b$ must have a common neighbor with $x_1$, otherwise we would have an induced $P_4$ staring from $x_1$ into $H_2$. Thus, there is $u$ such that $x_1ub$ is a path joining $x_1$ to $b$. Clearly, $u$ must lie inside $S_1$. Now, $u$ is adjacent to $a$, otherwise $x_2abu$ is an induce $P_4$ from $x_2$ into $H_2$, and hence $u$ is also adjacent to $x_3$, otherwise $x_3bau$ is an induced $P_4$ from $x_3$ into $H_2$. Similarly, any vertex in $S_1$ that is adjacent to $a$ or to $b$, must be adjacent to both of them.
Hence, every vertex in $S_1$ must lie in the region bounded by $abu$. So, if $S_1$ has more than one vertex, then $abu$ is a separating triangle. If it has at least two vertices in its interior, one of them is a non-neighbor of $u$, and hence there is an induced $P_3$ from $u$ into this region, which then together with $x_1u$ gives an induced $P_4$ from $x_1$ into $H_2$, a contradiction. Thus, $abu$ has at most one vertex in its interior, which (if exists) is adjacent to each of $a, b$ and $u$. Thus, $S_1$ contains at most two vertices, and hence $S_2$ must contain at least three vertices. 

Note that no vertex in $S_2$ can be adjacent to $b$, otherwise with $x_1ub$ they would form an induced $P_4$ from $x_1$ into $H_2$. If a vertex in $S_2$ is not adjacent to $a$, then there must be a path joining it to $a$ using only vertices inside $S_2$, otherwise $\{x_2,x_3\}$ would be a cut-set in $G$, a contradiction. Then, there is an induced $P_3$ from $a$ into $S_2$, which together with $x_1a$ gives an induced $P_4$ from $x_1$ into $H_2$, a contradiction. Thus, every vertex in $S_2$ is adjacent to $a$. Then, they should all be adjacent to $x_3$ as well, otherwise if $v$ inside $S_2$ is not adjacent to $x_3$, then  $x_3bav$ is an induced $P_4$ starting from $x_3$ into $H_2$, a contradiction. Now, $x_3$ must have a neighbor in $S_2$, otherwise $\{a,x_2\}$ is a cut-set. Also, if there are at least two neighbors of $x_2$ in $S_2$, then they would be common neighbors to all of $x_2, x_3$ and $a$, together with $x_1$, they give a copy of $K_{3,3}$ in $G$, a contradiction. Thus, $x_2$ has exactly one neighbor in $S_2$, say $v$. There are at least two other vertices in $S_2$ apart from $v$, as they are all adjacent to both $a$ and $x_3$, they must lie inside the region $S_3$ bounded by $avx_3b$. Recall that none of them are adjacent to $b$, and if none of them are adjacent to $v$, then again $\{a,x_3\}$ would be a cut. So, $v$ has a neighbor inside $S_3$, say $v'$. Then, $v'vx_2x_1u$ is an induced $P_5$ in $G$, a contradiction. This completes the proof if $G$ has a cut-set of size $3$, and has at least $11$ vertices.

Now, assume that $G$ is $4$-connected and has at least $12$ vertices. Let $X$ be a cut-set of size $4$. Let $v \in X$, then $G\setminus v$ is 3-connected, has $X\setminus \{v\}$ as a cut-set of size three, and has at least $11$ vertices. Then, by the above $G\setminus v$,  and hence $G$, contains an induced $P_5$.

Finally, let $G$ be $5$-connected on at least $13$ vertices, with a cut-set $X$ of size $5$. Then, $G \setminus v$ for some $v \in X$ is $4$-connected on at least $12$ vertices. Thus, by the previous paragraph, $G\setminus v$, and hence $G$ contains an induced $P_5$. This completes the proof.
\end{proof}

\begin{proof}[\textbf{Proof of Theorem \ref{P_5}}]
    If $3 \leq n\leq 8$, then we can have triangulations on $n$ vertices avoiding induced $P_5$: for $n\leq 5$, any triangulations, for $n=6,7$, take $\overline{K}_2+C_{n-2}$, and for $n=8$, any triangulation given in Figure \ref{fig:8trnoP5}. Thus, $\ex_\cP(n, P_5^\ind)=3n-6$, for every $3\leq n\leq 8$. 
    
    Let $n\geq 9$. To prove the lower bound, consider the graph $G$ defined as follows. Take the $8$-vertex triangulation $A$ given in figure \ref{fig:8trnoP5}, and for the rest of $n-8$ vertices take $\lfloor \frac{n-8}{4} \rfloor$ paths on four vertices and a shorter path on the remaining vertices. Join all these paths to the two bottom vertices in $A$. It is easy to see that $G$ is a $P_5^\ind$-free planar graph with $|E(G)|=\lfloor\frac{11n}{4}\rfloor-4$.
    
    Now, it is left to prove the upper bound. By Lemma \ref{>9tr>indP5}, any triangulation on $n$ vertices contains an induced $P_5$, and hence $\ex_\cP(n,P_5^\ind) \leq 3n-7$. Since $3n-7 \leq \lfloor\frac{11n}{4}\rfloor-4$, for every $n \leq 12$, then we are done for every $n \leq 12$.

    Assume that $n \geq 13$, and $G$ is a planar graph on $n$ vertices without containing an induced $P_5$. We will apply induction on $n$. By Proposition \ref{3conP_5ind}, $G$ is not $3$-connected. We may also assume that the minimum degree of $G$ is at least $3$, otherwise we are done by induction. If $G$ is not connected, then we can apply induction on its components. Let $X$ be a minimum cut-set of $G$, $H_1, \ldots, H_k$ be the components  $G\setminus X$. Let $G_i$ be $G[V(H_i) \cup \{v\}]$ and $n_i:= |V(G_i)|$. If $G$ has a cut vertex, i.e. $|X|=1$, then $ \sum_{n=1}^k n_i=n+k-1$. By induction hypothesis $|E(G_i)| \leq \lfloor\frac{11n_i}{4}\rfloor-4$ (note that for every $n \leq 8$, $3n-6 \leq \lfloor\frac{11n}{4}\rfloor-4$). Thus,
    \[
    |E(G)| \leq \sum_{i=1}^k |E(G_i)| \leq \sum_{i=1}^k \lfloor \frac{11n_i}{4}\rfloor-4k \leq \lfloor \frac{11(n+k-1)}{4}\rfloor-4k < \lfloor \frac{11n}{4}\rfloor-4
    \]


    So, we may assume that $G$ is $2$-connected. That is $|X|=2$, say $X=\{x,y\}$.  We may assume that $xy \in E(G)$, since adding the edge $xy$ does not create an induced $P_5$. To see this, assume $xy$ is not an edge but adding it creates an induced $P_5$. Let $P$ be such an induced $P_5$ after adding $xy$. Since $G$ was $P_5^\ind$-free, $xy$ must be an edge of $P$. Assume that $V(P)\subseteq V(G_i)$, for some $i=1,2, \ldots, k$. Then, if $xy$ is the last edge of $P$, say with $y$ being the last vertex, we can replace $xy$ by $xu$ for some neighbor of $u$ in some $H_j$ with $j\neq i$ and get an induced $P_5$ in $G$, a contradiction. If $xy$ is not the last edge of $P$, then we replace $xy$ by the shortest path joining $x$ and $y$ in another  $G_j$ with $j\neq i$ to get an induced path of length at least $6$ in $G$, a contradiction. If $V(P)$ uses vertices in different components, then it must be of the form $uxyvw$, where $u \in V(H_i)$ and $v,w \in V(H_j)$ for some $i \neq j \in \{1,2, \ldots , k\}$. Thus, $v$ and $w$ are not adjacent to $x$ and $u$ is not adjacent to $y$. Take the shortest path from $u$ to $y$ in $H_i$, it has length at least two. Thus, from $y$ we have an induced $P_3$ into each of $H_i$ and $H_j$, and hence an induced $P_5$ in $G$, a contradiction.
    
    Recall that by Observation \ref{obs1}, each of $x$ and $y$ must have a neighbor in each $H_i$ for every $i$, and by Observations \ref{obs5} and \ref{obs6}, 
    if $x$ (or $y$) has a non-neighbor in a component, then it is adjacent to every vertex in all the other components. 
    
    Thus, if $k \geq 3$, then there is a component $H_i$ in which every vertex is adjacent to both of $x$ and $y$. If a vertex $v \in V(H_i)$ has degree  at least $3$ in $H_i$, then $\{v, x, y\}$ would have three common neighbors, which forms a $K_{3,3}$ in $G$, a contradiction. Thus, $G[V(H_i)]$ is connected with maximum degree at most two, and hence it must be a path. Hence, if $|V(H_i)|\geq 5$, it would be an induced $P_5$ in $G$, a contradiction. Therefore, $|V(H_i)| \leq 4$ which forms a path and each vertex is adjacent to each of $x$ and $y$. Then, it is easy to check that the number of edges incident to $V(H_i)$ is at most $\lfloor \frac{11}{4} |V(H_i)|\rfloor$. We can now apply induction on $G\setminus V(H_i)$ and we are done.

    Assume $G$ has at most two components, i.e. $k=2$. If both $x$ and $y$ have non-neighbors in a component, then they are both adjacent to every vertex in the other, and we are done as above. So assume that every vertex of $H_1$ is adjacent to $x$ but $y$ has non-neighbors in $H_1$, and  every vertex in $H_2$ is adjacent to $y$ but $x$ has non-neighbors in $H_2$.

    We show that $H_1$ (and similarly $H_2$) has at most $5$ vertices or contains a set $S$ incident with at most $\lfloor \frac{11}{4}|S|\rfloor$ edges, and we are done by induction. Let $Y$ be the set of non-neighbors of $y$ in $H_1$, and $C=V(H_1)\setminus Y$. Then, every vertex in $C$ is a common neighbor to both $x$ and $y$. Note that a vertex $y_1\in Y$ cannot have more than two common neighbors with $y$, otherwise we would have three common neighbors of $y_1, y$ and $x$, which means a $K_{3,3}$ in $G$, a contradiction.

    If $|Y|=1$, say $y_1$ is the only non-neighbor of $y$. Then it is adjacent to $x$ and must be adjacent to two other vertices, say $c_1$ and $c_2$, otherwise it would have degree at most two. 
    If $G[C]$ contains a cycle, then together with $xy$ they form a subdivision of $K_5$, a contradiction. Then, any path in $G[C]$ is an induced path. Thus, $G[C]$ consists of disjoint paths on at most $4$ vertices. Thus, if $G[C]$ is connected, then it has at most $4$ vertices, and hence $H_1$ has at most $5$ vertices. If $G[C]$ has at least two components, then either $H_2= G[Y \cup C]$ is not connected, which is a contradiction, or the components of $G[C]$ are connected through $y_1$ (the vertex in $Y$). But $y_1$ has degree $2$ in $C$, and hence $G[C]$ can have at most two components. If $|C|\geq 5$, then at least one component of $C$ is an induced $P_3$, and all of its vertices are adjacent to both $x$ and $y$. Hence the neighbor of $y_1$ in this component must an end-vertex of this induced $P_3$, and then using $y_1$ and its neighbor in the other component of $G[C]$, we can get an induced $P_5$ in $G$, a contradiction.
    
    Now assume that $|Y| \geq 2$. If a vertex $y_1\in Y$ has two common neighbors $c_1$ and $c_2$ with $y$, then it cannot be adjacent to any other vertex of $H_1$. Since it can have at most two neighbors in $C$, and if it has a neighbor $y_2\in Y$, then either $y_2y_1c_1y$ or $y_2y_1c_2y$ is an induced $P_4$ from $y$ into $H_1$, or $y_2$ is adjacent to both of $c_1$ and $c_2$. However, in the latter case $y_1, y_2, c_1, c_2$ and $y$ form a subdivision of $K_4$, and hence with $x$, they form a subdivision of $K_5$ in $G$, a contradiction. Also, any other vertex of $Y\setminus \{y_1\}$ is either adjacent to $c_1$ or $c_2$ but not both. Then, any vertex of $Y \setminus \{y_1\}$ is either a common neighbor of $x$ and $c_1$ or a common neighbor of $x$ and $c_2$. Hence, using the same arguments as before, we find a set $S$ among the common neighbors of $x$ and $c_1$ (or $x$ and $c_2$) incident with $\lfloor \frac{11}{4}|S|\rfloor$ edges, and we are done by induction again.

    If no vertex of $Y$ has two common neighbors with $y$ (which will be in $C$), then all vertices of $Y$ are connected to $y$ through one same common neighbor $c$ in $C$. Otherwise, assume that  $Y_1 \subseteq Y$ are the vertices of $Y$ that are connected to $y$ through $c_1 \in C$, and $Y_2\subseteq Y$ be those that are connected to $y$ through $c_2 \in C$, with $c_1 \neq c_2$. Since no vertex in $Y$ has two neighbors in $C$, then no vertex in $Y_1$ is adjacent to $c_2$ and no vertex of $Y_2$ is adjacent to $c_1$. Also, if a vertex $y_1\in Y_1$ is adjacent to a vertex in $y_2 \in Y_2$, then $yc_1y_1y_2$ is an induced $P_4$ from $y$ into $H_1$, a contradiction. Then every vertex of $Y_1$ is a common neighbor of $c_1$ and $x$, and every  vertex in $Y_2$ is a common neighbor of $c_2$ and $x$, and hence similarly to the previous cases we can find the desired set $S$. 

    Thus, we either find a desired set $S$ in $H_1$ or $H_2$, and we are done by induction, or both $H_1$ and $H_2$ has at most $5$ vertices. In this case, $G$ has at most $12$ vertices and we are done.
    \end{proof}

\section{Concluding remarks}
As noted in the introduction for every $k\geq 5$, we have $\ex_\cP(n, \Theta_k^\ind)=3n-6$. Another potential class of graphs to consider is the cycles with all chords added (so they would not contain an induced $C_4$), that is, triangulated $n$-gons. Let $C_{n,2}$ denote the $5$-cycle with its two diagonal edges added to it. We conjecture that $\ex_\cP(n, C_{5,2}^\ind)\leq \frac{48}{17}(n-2)$ for every $n \geq 5$. Note that replacing each $K_4$ block in the construction in Figure \ref{figrTheta4}, by the $6$-vertex triangulation $\overline{K_2}+C_4$, one can obtain a construction for every $n\equiv 36 \ (mode \ 68)$ that is $C_{5,2}^\ind$-free and attains this conjectured bound.

A more challenging problem is to consider forbidding induced paths $P_k$, for $k \geq 6$.  For the case $k=6$, we conjecture the following.
\begin{conj}\label{P_6} For every $n \geq 3$,

    $\ex_\cP(n,P_6^\ind)=\left\{
\begin{array}{lll}
      3n-6 & , & n\leq 19 \\
      \lfloor\frac{50}{17}(n-2)\rfloor+1 & , & n\geq 20
      \end{array} \right.$
\end{conj} 

Note that the $19$-vertex triangulation given in Figure \ref{fig:9trnoP6} does not contain an induced $P_6$. Also, repeatedly deleting a vertex of degree $3$ from it we can obtain a triangulation for each smaller value of $n$ that is  $P_6^\ind$-free. For $n\geq 20$, The following construction is $P_6^\ind$-free and achieves this value. Take $\lfloor \frac{n-2}{17} \rfloor$ copies of the given triangulation below with the edge $ab$ shared among all of them, and another triangulation on the left over vertices together with $a$ and $b$ obtained from the given one by deleting the degree $3$ vertices repeatedly, such that it shares the edge $ab$, too.

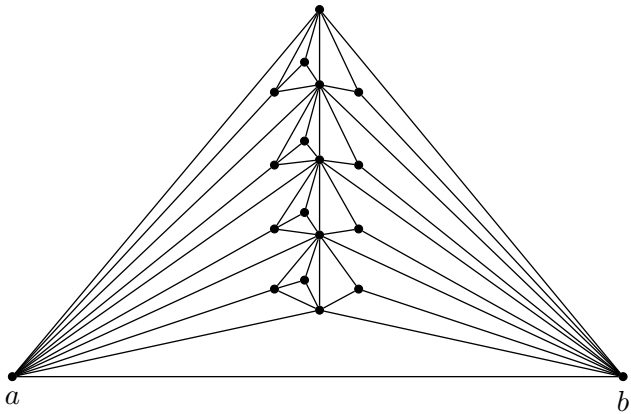
\begin{figure}[h]
        \centering
\begin{center}
    \begin{tikzpicture}
    \node[vertex, label=below:$a$] (a1) {};
    \node[vertex, label=below:$b$] (b1) [right=8cm of a1] {};
\node [vertex] (c1) [above right=0.8cm and 4cm of a1] {};
\node [vertex] (d1) [above right=1.8cm and 4cm of a1] {};
\node [vertex] (e1) [above right=2.8cm and 4cm of a1] {};
\node [vertex] (f1) [above right=3.8cm and 4cm of a1] {};
\node [vertex] (g1) [above right=4.8cm and 4cm of a1] {};
\node [vertex] (h1) [above right=3.7cm and 3.4cm of a1] {};
\node [vertex] (h2) [below=0.85cm of h1] {};
\node [vertex] (h3) [below=1.7cm of h1] {};
\node [vertex] (h4) [below=2.5cm of h1] {};

\node[vertex] (i1) [right=1cm of h1] {};
\node[vertex] (i2) [right=1cm of h2] {};
\node[vertex] (i3) [right=1cm of h3] {};
\node[vertex] (i4) [right=1cm of h4] {};

\node[vertex] (j1) [above right=4.1cm and 3.8 of a1] {};
\node[vertex] (j2) [above right=3.05cm and 3.8 of a1] {};
\node[vertex] (j3) [above right=2.1cm and 3.8 of a1] {};
\node[vertex] (j4) [above right=1.2cm and 3.8 of a1] {};

\draw[line width=0.5pt] (a1) -- (b1) -- (c1)-- (a1)-- (d1) -- (b1) -- (e1) -- (a1) -- (f1) -- (b1) -- (g1)--(a1);
\draw[line width=0.5pt] (c1) -- (d1)-- (e1) -- (f1) -- (g1);
\draw[line width=0.5pt] (g1) -- (h1)-- (f1) -- (h2) -- (e1) -- (h3) -- (d1) -- (h4) -- (c1) -- (i4) -- (d1) -- (i3) -- (e1) -- (i2) -- (f1) -- (i1) -- (g1);
\draw[line width=0.5pt] (g1) -- (j1)-- (h1) -- (a1) -- (h2) -- (j2) -- (e1) -- (j3) -- (h3) -- (a1) -- (h4) -- (j4) -- (c1);
\draw[line width=0.5pt] (j1) -- (f1) -- (j2);
\draw[line width=0.5pt] (j3) -- (d1) -- (j4);
\draw[line width=0.5pt] (i1) -- (b1) -- (i2);
\draw[line width=0.5pt] (i3) -- (b1) -- (i4);
\end{tikzpicture}
 \end{center}   
        \caption{A $19$-vertex triangulation with no induced $P_6$.}
        \label{fig:9trnoP6}
    \end{figure}

\end{document}